\documentclass{amsart}
\usepackage{latexsym,amsfonts,amssymb,exscale,enumerate,comment}
\usepackage{amsthm,amscd}
\usepackage{amsmath}
\usepackage{hyperref}

\usepackage{url}

\input xy
\usepackage[all]{xy}
\xyoption{line}
\xyoption{arrow}
\xyoption{color}
\SelectTips{cm}{}
\usepackage{tikz}
\usetikzlibrary{shapes,snakes}
\usetikzlibrary{decorations.markings}
\usetikzlibrary{decorations.pathreplacing}
\tikzstyle directed=[postaction={decorate,decoration={markings,mark=at position #1 with {\arrow{>}}}}]

\newcommand{\hackcenter}[1]{\xy (0,0)*{#1}; \endxy}

\tikzset{->-/.style={decoration={markings, mark=at position #1 with {\arrow{>}}},postaction={decorate}}}

\tikzset{middlearrow/.style={ decoration={markings,mark= at position 0.5 with {\arrow{#1}} , },
postaction={decorate}}}

\usepackage{graphicx}
\usepackage{color}

\def\mf{\mathfrak}
\newcommand{\scs}{\scriptstyle}

\theoremstyle{plain}
\newtheorem{theorem}{Theorem}

\newtheorem{corollary}[theorem]{Corollary}
\newtheorem{proposition}[theorem]{Proposition}
\newtheorem{lemma}[theorem]{Lemma}
\theoremstyle{definition}
\newtheorem{example}[theorem]{Example}
\newtheorem{definition}[theorem]{Definition}

\newtheorem{remark}[theorem]{Remark}

\numberwithin{equation}{section}
\numberwithin{theorem}{section}

\renewcommand{\to}{\rightarrow}

\numberwithin{equation}{section}
\let\hat=\widehat

\usepackage{bbm}

\def\Z{{\mathbbm Z}}

\title{A $p$-DG deformed Webster algebra of type $A_1$}
\author{Yasuyoshi Yonezawa}
\email{yasuyoshi.yonezawa@math.nagoya-u.ac.jp}
\address{Graduate school of Mathematics \\
Nagoya University\\
Nagoya, Japan}

\begin{document}
\begin{abstract}
We define a $p$-DG structure on a deformation of Webster algebra of type $A_1$ and its splitter bimodules.
\end{abstract}

\maketitle

\section{Introduction}

Since Khovanov categorified the Jones polynomial \cite{Kh1}, a categorification of the quantum link invariant of type $A$ has been constructed in several constructions.
More precisely, our motivation is to construct a link homology which is an invariant of the link and whose graded Euler characteristic is the quantum link invariants.
In the case of the categorification for type $A$, we have a construction of matrix factorizations \cite{KR,Wu, Yo}, a geometric construction \cite{CK-slm}, a Lie theoretic construction \cite{Su,MS-O}, a diagrammatic construction \cite{Web} and a Howe duality construction \cite{CKL-skew, CK-symm, MY}.


We are also interested in a categorification of the quantum invariant for three-dimensional manifolds.
Defining the quantum link invariant we need quantum group for the generic parameter $q$.
When we construct quantum invariants for three dimensional manifold, the parameter of quantum group should be a root of unity $\zeta$.
One direction to construct the categorification of the WRT three manifold invariants is that, first, we categorify the quantum groups at a root of unity and, then, we extend categorical Howe duality into the categorification of quantum group of root of unity.

Khovanov proposed the categorification of a root of unity.
He introduced \emph{hopfological algebra} theory which is theory of Hopf algebra with homological algebra. See \cite{Kh4,Q1} in detail.
The key observation to categorify a root of unity is that the ring $\Z[\zeta_p]$, where $\zeta_p$ is the $p$-cyclotomic integer, is categorified by the homotopy category of $p$-complexes $\Bbbk[\partial]/(\partial^p)$-$\underline{\mathrm{gmod}}$ over a field of characteristic $p>0$:
$$
K_0(\Bbbk[\partial]/(\partial^p)\text{-}\underline{\mathrm{gmod}})\cong \Z[\zeta_p].
$$
In order to apply the categorification of the ring $\Z[\zeta_p]$ to a categorification of WRT invariants, we need algebra objects in the monoidal category $\Bbbk[\partial]/(\partial^p)$-$\underline{\mathrm{gmod}}$ which induce a braiding functor.
Such objects naturally arise from $p$-differential graded algebras (\emph{$p$-DG algebras}).

To solve the open problem of categorification of WRT invariants, we have some progress results about $p$-DG structure\cite{KQ,KQYpDG,QSus,EQ,EQ2,QYpDG,EllisQ}.
For instance, we have the $p$-DG structure on the symmetric polynomial ring, cyclotomic nilHecke algebra $\mathrm{NH^l_n}$ (a categorification of an irreducible representation of small quantum group for $sl_2$), quiver Schur algebra $S_n^l$ (a categorification of the tensor product $V^{\otimes \ell}_1$, where $V_1$ is the Weyl module of small quantum group $sl_2$), Webster algebras.

In this short note, we briefly recall $p$-DG algebras, $p$-DG modules\cite{Q1}, some progress result, the deformed Webster algebras of type $A_1$\cite{KLSY}, and then we define a $p$-DG structure on the deformed Webster algebra.

\section{$p$-DG algebras and $p$-DG modules}
Let $\Bbbk$ be a field of characteristic $p>0$.
We recall $p$-DG algebras and $p$-DG modules defined in \cite{KQ} and give some $p$-DG algebras introduced in \cite{KQYpDG,KQ,EQ2,QYpDG}.
\subsection{$p$-DG algebras and $p$-DG modules}
\begin{definition}
A $\Z$-graded  algebra $A$ is said to be a $p$-DG algebra if there is a $p$-derivation $\partial_A:A\to A$ with degree $2$, which is $p$-nilpotent and satisfies the Leibniz rule:
$$
\partial_A^p=0,\quad \partial_A(ab) = \partial_A(a)b + a\partial_A(b)
$$
for any $a, b \in A$.
\end{definition}
We will write $(A,\partial_A)$ for the $p$-DG algebra.

By the Leibniz rule and the characteristic of the field $p$, we have the following lemma.
\begin{lemma}\label{lemma}
We have
$$
\partial_A^p(ab) = \partial_A^p(a)b + a\partial_A^p(b).
$$
\end{lemma}
When the $p$-DG algebra is generated by the elements of the set $\{a_i|i\in I\}$, to check $\partial_A^p=0$, it is enough to show $\partial_A^p(a_i)=0$ for all $i \in I$ by this lemma.

\begin{definition}
Let $(A,\partial_A)$ and $(B,\partial_B)$ be $p$-DG algebras.
A $\Z$-graded $A$-module $M$ is said to be a left $p$-DG module if there is a degree $2$ endomorphism $\partial_M$ satisfying
\[
\partial_M^p(m)=0, \quad \quad \partial_M(	am)=\partial_A(a)m+a\partial_M(m)
\]
for any elements $a\in A$ and $m\in M$.

Similarly, a $\Z$-graded $A$-module $M$ is said to be a right $p$-DG module if there is a degree $2$ endomorphism $\partial_M$ satisfying
\[
\partial_M^p(m)=0, \quad \quad \partial_M(	ma)=m\partial_A(a)+\partial_M(m)a.
\]
for any elements $a\in A$ and $m\in M$.

A $\Z$-graded $(A,B)$-bimodule $M$ is said to be a $p$-DG bimodule if $M$ is a left $p$-DG $A$-module and a right $p$-DG $B$-module.
\end{definition}
We will write $(M,\partial_M)$ for the $p$-DG modules.

\subsection{$p$-DG structure on the polynomial ring}
We recall $p$-DG structure on the graded polynomial ring\cite{EQ}.

The graded polynomial ring $R=\Bbbk [t_1,...,t_n]$, where $\deg t_i=2$, is a $p$-DG algebra with a $p$-derivation
$$
\xymatrix@R=2pt{
\partial_R:&R\ar[r]\ar@{}[d]|-{\rotatebox{90}{$\in$}}& R\ar@{}[d]|-{\rotatebox{90}{$\in$}}\\
&t_i\ar@{|->}[r]&t_i^2.
}
$$
satisfying the Leibniz rule.
By the Leibniz rule, we have $\partial_R(t_i^k)=k t_i^{k+1}$, and thus we have $\partial_R^p(t_i)=p! t_i^{p+1}=0$.
Hence, we have the following statement by Lemma \ref{lemma}.
\begin{proposition}
The polynomial ring with the $p$-derivation $(R,\partial_R)$ is a $p$-DG algebra.
\end{proposition}
%
%
%

The invariant ring $R^{\mathfrak{S}_n}$ is generated by the elementary symmetric functions
$$
E_i=\sum_{1 \leq j_1<j_2<\cdots <j_i\leq n}t_{j_1}t_{j_2}\cdots t_{j_i}\qquad (i=1,...,n).
$$

We set the generating function
$$
G:=\sum_{i=0}^n E_i s^i=\prod_{i=1}^n(1+t_i s),
$$
where $E_0=1$.
We have
\begin{eqnarray*}
\partial_R(G)=\sum_{i=1}^n \partial_R(E_i) s^i
&=&\sum_{j=1}^{n}\partial_R(t_j)s\prod_{i=1\atop i\not= j}^n(1+t_i s)
=\sum_{j=1}^{n}t_j^2s\prod_{i=1\atop i\not= j}^n(1+t_i s)\\
&=&\sum_{j=1}^{n}t_j\prod_{i=1}^n(1+t_i s)-\sum_{j=1}^{n}t_j\prod_{i=1\atop i\not= j}^n(1+t_i s)\\
&=&E_1G-\frac{d}{ds}(G)=\sum_{i=1}^n (E_1E_i-(i+1)E_{i+1}) s^i,
\end{eqnarray*}
where $E_{n+1}=0$.

Therefore, the $p$-derivation $\partial_R$ acts on the elementary symmetric functions $E_i$ as
$$
\partial_R(E_i)=E_1E_i-(i+1)E_{i+1} \text{ for } 1\leq i\leq n-1, \quad \partial_R(E_n)=E_1E_n.
$$

Using Lemma \ref{lemma} we have
\begin{eqnarray*}
\partial_R^p(G)&=&\sum_{i=0}^n \partial_R^p(E_i) s^i=\partial_R^p\left(\prod_{i=1}^n(1+t_i s)\right)=0
\end{eqnarray*}
since we have $\partial_R^p(1+t_i s)=\partial_R^p(t_i )s=0$.
Hence, we have the following statement.
\begin{proposition}\label{symfun}
The invariant ring with the $p$-derivation $(R^{\mathfrak{S}_n},\partial_R)$ is a $p$-DG algebra.
\end{proposition}

Proving existence of $p$-DG structure on the deformed Webster algebra in Section \ref{sec4}, we show $\partial_R^p(E_i)=0$ in detail.

By Leibniz rule, we have
\begin{eqnarray*}
\partial_R^p(G)&=&
\sum_{\underline{p}=(p_1,p_2,...,p_n)\in \Z_{\geq 0}^n\atop p_1+p_2+\cdots +p_n=p}
\left(p\atop \underline{p}\right)\left(\prod_{i=1}^n\partial_R^{p_i}(1+t_i s)\right)
=
\sum_{\underline{p}=(p_1,p_2,...,p_n)\in \Z_{\geq 0}^n\atop p_1+p_2+\cdots +p_n=p}
\left(p\atop \underline{p}\right)\left(\prod_{i=1}^n(\partial_R^{p_i}(1)+p_i! t_i^{p_i+1} s)\right),
\end{eqnarray*}
where
$$
\left(p\atop \underline{p}\right)=\prod_{a=0}^{n-1} \left(p-\sum_{b=1}^{a}p_b\atop p_{a+1}\right).
$$

Since the characteristic of $\Bbbk$ is $p$  we have
$$
\left(p\atop \underline{p}\right)p_i!=
\left(p\atop p_i\right)\left(p-p_i\atop \underline{p}(i)\right)p_i!=
p\left(p-1\atop p_i-1\right)\left(p-p_i\atop \underline{p}(i)\right)(p_i-1)!=0,
$$
where $\underline{p}(i)=(p_1,...,p_{i-1},p_{i+1},...,p_n)\in \Z_{\geq 0}^{n-1}$.
Therefore, each term in $\partial_R^p(G)$ is divisible by $p$.
That is, each term in $\partial_R^p(E_i)$ is divisible by $p$.
Hence, we have $\partial_R^p(E_i)=0$.

\subsection{The $p$-DG nilHecke algebra}
We recall the nilHecke algebra and its diagrammatic description\cite{Lau1}.

The nilHecke algebra $\mathrm{NH}_n$ is the unital algebra generated by $x_i$ ($i=1,...,n$) and $\psi_j$ ($j=1,...,n-1$) subject to relations
\begin{eqnarray}
& & \psi_i^2=0, \ \  \psi_i \delta_{i+1} \psi_i =
\psi_{i+1}\psi_i \psi_{i+1}, \ \  \psi_i\psi_j = \psi_j
\psi_i
\ \ \mathrm{if} \ \  |i-j|>1, \\
& &  x_i \psi_j = \psi_j x_i \ \ \mathrm{if} \ \  j\not= i, i+1, \ \
 x_i x_j = x_j x_i, \\
 & & x_i \psi_i - \psi_i x_{i+1} = 1, \ \
     \psi_i x_i - x_{i+1}\psi_i = 1.
\end{eqnarray}

We have a graphical presentation of $\mathrm{NH}_n$ whose the unit $1_{\mathrm{NH}_n}$ is depicted as the $n$-strands, $x_i$ is depicted as a dot on the $i$-th strand and $\psi_i$ is the crossing of the $i$-th and $(i+1)$-st strands:

\[
1_{\mathrm{NH}_n}=
\hackcenter{\begin{tikzpicture}[scale=0.75]
    \draw[thick, ] (0,0) -- (0,1.5) ;
    \draw[thick, ] (.75,0) -- (.75,1.5) ;
    \node at (1.625,.75) {$\cdots$};
    \draw[thick, ] (2.5,0) -- (2.5,1.5) ;
    \node at (0,-.25) {$\scs 1$};
    \node at (.75,-.25) {$\scs 2$};
    \node at (2.5,-.25) {$\scs n$};
\end{tikzpicture}}
\quad \quad \quad \quad
x_i:=
\hackcenter{\begin{tikzpicture}[scale=0.75]
    \draw[thick, ] (0,0) -- (0,1.5) ;
    \node at (.75,.75) {$ \cdots$};
    \draw[thick, ] (1.5,0) -- (1.5,1.5) ;
    \node at (2.25,.75) {$ \cdots$};
    \draw[thick, ] (3,0) -- (3,1.5) ;
    \node at (0,-.25) {$\scs 1$};
    \node at (1.5,-.25) {$\scs i$};
    \node at (3,-.25) {$\scs n$};
    \filldraw  (1.5,.75) circle (2.5pt);
\end{tikzpicture}}
\quad \quad \quad \quad
\psi_i:=
\hackcenter{\begin{tikzpicture}[scale=0.75]
    \draw[thick, ] (-.5,0) -- (-.5,1.5) ;
    \node at (.25,.75) {$\cdots$};
    \draw[thick, ] (1,0) .. controls (1,.75) and (1.5,.75) .. (1.5,1.5);
    \draw[thick, ] (1.5,0) .. controls (1.5,.75) and (1,.75) .. (1,1.5);
    \node at (2.25,.75) {$\cdots$};
    \draw[thick, ] (3,0) -- (3,1.5) ;
    \node at (-.5,-.25) {$\scs 1$};
    \node at (1,-.25) {$\scs i$};
    \node at (1.5,-.25) {$\scs i+1$};
    \node at (3,-.25) {$\scs n$};
\end{tikzpicture}}
\]


Multiplication in $\mathrm{NH}_n$ is represented by vertical concatenation of diagrams.
The product $xy$ is represented as follows:
\[
\hackcenter{\begin{tikzpicture}[scale=0.75]
    \draw[thick, ] (0,0) -- (0,1.5) ;
    \draw[thick, ] (.75,0) -- (.75,1.5) ;
    \node at (1.625,.1) {$\cdots$};
    \node at (1.625,1.35) {$\cdots$};
    \draw[thick, ] (2.5,0) -- (2.5,1.5) ;
    \filldraw[color=white,draw=black] (-.25,.3) rectangle (2.75,1.2);
    \node at (0,-.25) {$\scs 1$};
    \node at (.75,-.25) {$\scs 2$};
    \node at (2.5,-.25) {$\scs n$};
    \node at (2.5,1.75) {};
    \node at (1.25,.75) {$x$};
\end{tikzpicture}}
\cdot
\hackcenter{\begin{tikzpicture}[scale=0.75]
    \draw[thick, ] (0,0) -- (0,1.5) ;
    \draw[thick, ] (.75,0) -- (.75,1.5) ;
    \node at (1.625,.1) {$\cdots$};
    \node at (1.625,1.35) {$\cdots$};
    \draw[thick, ] (2.5,0) -- (2.5,1.5) ;
    \filldraw[color=white,draw=black] (-.25,.3) rectangle (2.75,1.2);
    \node at (0,-.25) {$\scs 1$};
    \node at (.75,-.25) {$\scs 2$};
    \node at (2.5,-.25) {$\scs n$};
    \node at (2.5,1.75) {};
    \node at (1.25,.75) {$y$};
\end{tikzpicture}}
\;\; = \;\;
\hackcenter{\begin{tikzpicture}[scale=0.75]
    \draw[thick, ] (0,0) -- (0,1.5) ;
    \draw[thick, ] (.75,0) -- (.75,1.5) ;
    \node at (1.625,.75) {$\cdots$};
    \draw[thick, ] (2.5,0) -- (2.5,1.5) ;
    \node at (1.625,.05) {$\cdots$};
    \node at (1.625,1.4) {$\cdots$};
    \filldraw[color=white,draw=black] (-.25,.9) rectangle (2.75,1.3);
    \filldraw[color=white,draw=black] (-.25,.2) rectangle (2.75,.6);
    \node at (0,-.25) {$\scs 1$};
    \node at (.75,-.25) {$\scs 2$};
    \node at (2.5,-.25) {$\scs n$};
    \node at (2.5,1.75) {};
    \node at (1.25,1.1) {$\scs x$};
    \node at (1.25,.4) {$\scs y$};
\end{tikzpicture}}
\]

The defining relations are represented by the following diagrammatic equalities.

$$
\hackcenter{\begin{tikzpicture}[scale=0.75]
    \draw[thick] (0,0) .. controls ++(0,.375) and ++(0,-.375) .. (.75,.75);
    \draw[thick] (.75,0) .. controls ++(0,.375) and ++(0,-.375) .. (0,.75);
    \draw[thick] (0,.75 ) .. controls ++(0,.375) and ++(0,-.375) .. (.75,1.5);
    \draw[thick] (.75,.75) .. controls ++(0,.375) and ++(0,-.375) .. (0,1.5);
\end{tikzpicture}}
=0, \ \ \ \ \ \ \ \
\hackcenter{\begin{tikzpicture}[scale=0.75]
    \draw[thick,  ] (0,0) .. controls ++(0,1) and ++(0,-1) .. (1.2,2);
    \draw[thick] (.6,0) .. controls ++(0,.5) and ++(0,-.5) .. (0,1.0);
    \draw[thick] (0,1.0) .. controls ++(0,.5) and ++(0,-.5) .. (0.6,2);
    \draw[thick] (1.2,0) .. controls ++(0,1) and ++(0,-1) .. (0,2);
    \node at (0,2) {$\scs \:$};
\end{tikzpicture}}
\;\; = \;\;
\hackcenter{\begin{tikzpicture}[scale=0.75]
    \draw[thick] (0,0) .. controls ++(0,1) and ++(0,-1) .. (1.2,2);
    \draw[thick] (.6,0) .. controls ++(0,.5) and ++(0,-.5) .. (1.2,1.0);
    \draw[thick] (1.2,1.0) .. controls ++(0,.5) and ++(0,-.5) .. (0.6,2.0);
    \draw[thick] (1.2,0) .. controls ++(0,1) and ++(0,-1) .. (0,2.0);
    \node at (0,2) {$\scs \:$};
\end{tikzpicture}}
$$

$$
\hackcenter{\begin{tikzpicture}[scale=0.75]
    \draw[thick] (0,0) .. controls (0,.75) and (.75,.75) .. (.75,1.5)
        node[pos=.25, shape=coordinate](DOT){};
    \draw[thick] (.75,0) .. controls (.75,.75) and (0,.75) .. (0,1.5);
    \filldraw  (DOT) circle (2.5pt);
\end{tikzpicture}}
\quad -\quad
\hackcenter{\begin{tikzpicture}[scale=0.75]
    \draw[thick] (0,0) .. controls (0,.75) and (.75,.75) .. (.75,1.5)
        node[pos=.75, shape=coordinate](DOT){};
    \draw[thick] (.75,0) .. controls (.75,.75) and (0,.75) .. (0,1.5);
    \filldraw  (DOT) circle (2.5pt);
\end{tikzpicture}}
\quad =\quad
\hackcenter{\begin{tikzpicture}[scale=0.75]
    \draw[thick] (0,0) -- (0,1.5);
    \draw[thick] (.75,0) -- (0.75,1.5);
\end{tikzpicture}}
\quad =\quad
\hackcenter{\begin{tikzpicture}[scale=0.75]
    \draw[thick] (0,0) .. controls (0,.75) and (.75,.75) .. (.75,1.5)
        node[pos=.75, shape=coordinate](DOT){};
    \draw[thick] (.75,0) .. controls (.75,.75) and (0,.75) .. (0,1.5);
    \filldraw  (DOT) circle (2.5pt);
\end{tikzpicture}}
\quad -\quad
\hackcenter{\begin{tikzpicture}[scale=0.75]
    \draw[thick] (0,0) .. controls (0,.75) and (.75,.75) .. (.75,1.5)
        node[pos=.25, shape=coordinate](DOT){};
    \draw[thick] (.75,0) .. controls (.75,.75) and (0,.75) .. (0,1.5);
    \filldraw  (DOT) circle (2.5pt);
\end{tikzpicture}}
$$

We denote the element $x_i^{a}$ by a dot with a label $a$ on the $i$-th strand. 

\[
x_i^a:=
\hackcenter{\begin{tikzpicture}[scale=0.75]
    \draw[thick, ] (0,0) -- (0,1.5) ;
    \node at (.75,.75) {$ \cdots$};
    \draw[thick, ] (1.5,0) -- (1.5,1.5) ;
    \node at (2.25,.75) {$ \cdots$};
    \draw[thick, ] (3,0) -- (3,1.5) ;
    \node at (0,-.25) {$\scs 1$};
    \node at (1.5,-.25) {$\scs i$};
    \node at (3,-.25) {$\scs n$};
    \node at (1.75,.6) {$\scs a$};
    \filldraw  (1.5,.6) circle (2.5pt);
\end{tikzpicture}}
\]

We recall a $p$-DG structure on the nilHecke algebra\cite{KQYpDG}.

The nilHecke algebra $\mathrm{NH}_n$ is a $p$-DG algebra with a $p$-derivation for the generators
\begin{eqnarray*}
\partial_{\mathrm{NH}}\left(
\hackcenter{\begin{tikzpicture}[scale=0.75]
    \draw[thick, ] (0,0) -- (0,1.5) ;
    \node at (.75,.75) {$ \cdots$};
    \draw[thick, ] (1.5,0) -- (1.5,1.5) ;
    \node at (2.25,.75) {$ \cdots$};
    \draw[thick, ] (3,0) -- (3,1.5) ;
    \node at (0,-.25) {$\scs 1$};
    \node at (1.5,-.25) {$\scs i$};
    \node at (3,-.25) {$\scs n$};
    \filldraw  (1.5,.6) circle (2.5pt);
    \node at (0,1.75) {};
\end{tikzpicture}}
\right)
&:=&
\hackcenter{\begin{tikzpicture}[scale=0.75]
    \draw[thick, ] (0,0) -- (0,1.5) ;
    \node at (.75,.75) {$ \cdots$};
    \draw[thick, ] (1.5,0) -- (1.5,1.5) ;
    \node at (2.25,.75) {$ \cdots$};
    \draw[thick, ] (3,0) -- (3,1.5) ;
    \node at (0,-.25) {$\scs 1$};
    \node at (1.5,-.25) {$\scs i$};
    \node at (3,-.25) {$\scs n$};
    \node at (1.75,.6) {$\scs 2$};
    \filldraw  (1.5,.6) circle (2.5pt);
    \node at (0,1.75) {};
\end{tikzpicture}}
\\
\partial_{\mathrm{NH}}\left(
\hackcenter{\begin{tikzpicture}[scale=0.75]
    \draw[thick, ] (-.5,0) -- (-.5,1.5) ;
    \node at (.25,.75) {$\cdots$};
    \draw[thick, ] (1,0) .. controls (1,.75) and (1.5,.75) .. (1.5,1.5);
    \draw[thick, ] (1.5,0) .. controls (1.5,.75) and (1,.75) .. (1,1.5);
    \node at (2.25,.75) {$\cdots$};
    \draw[thick, ] (3,0) -- (3,1.5) ;
    \node at (-.5,-.25) {$\scs 1$};
    \node at (1,-.25) {$\scs i$};
    \node at (1.5,-.25) {$\scs i+1$};
    \node at (3,-.25) {$\scs n$};
    \node at (0,1.75) {};
\end{tikzpicture}}
\right)
&:=&
-
\hackcenter{\begin{tikzpicture}[scale=0.75]
    \draw[thick, ] (-.5,0) -- (-.5,1.5) ;
    \node at (.25,.75) {$\cdots$};
    \draw[thick, ] (1,0) .. controls (1,.75) and (1.5,.75) .. (1.5,1.5);
    \draw[thick, ] (1.5,0) .. controls (1.5,.75) and (1,.75) .. (1,1.5)  node[pos=.8, shape=coordinate](DOT){};
    \filldraw  (DOT) circle (2.5pt);
    \node at (2.25,.75) {$\cdots$};
    \draw[thick, ] (3,0) -- (3,1.5) ;
    \node at (-.5,-.25) {$\scs 1$};
    \node at (1,-.25) {$\scs i$};
    \node at (1.5,-.25) {$\scs i+1$};
    \node at (3,-.25) {$\scs n$};
    \node at (0,1.75) {};
\end{tikzpicture}}
-
\hackcenter{\begin{tikzpicture}[scale=0.75]
    \draw[thick, ] (-.5,0) -- (-.5,1.5) ;
    \node at (.25,.75) {$\cdots$};
    \draw[thick, ] (1,0) .. controls (1,.75) and (1.5,.75) .. (1.5,1.5);
    \draw[thick, ] (1.5,0) .. controls (1.5,.75) and (1,.75) .. (1,1.5)  node[pos=.2, shape=coordinate](DOT){};
    \filldraw  (DOT) circle (2.5pt);
    \node at (2.25,.75) {$\cdots$};
    \draw[thick, ] (3,0) -- (3,1.5) ;
    \node at (-.5,-.25) {$\scs 1$};
    \node at (1,-.25) {$\scs i$};
    \node at (1.5,-.25) {$\scs i+1$};
    \node at (3,-.25) {$\scs n$};
    \node at (0,1.75) {};
\end{tikzpicture}}.
\end{eqnarray*}

By the Leibniz rule, we have
$$
\partial_{\mathrm{NH}}^p\left(
\hackcenter{\begin{tikzpicture}[scale=0.75]
    \draw[thick, ] (0,0) -- (0,1.5) ;
    \node at (.75,.75) {$ \cdots$};
    \draw[thick, ] (1.5,0) -- (1.5,1.5) ;
    \node at (2.25,.75) {$ \cdots$};
    \draw[thick, ] (3,0) -- (3,1.5) ;
    \node at (0,-.25) {$\scs 1$};
    \node at (1.5,-.25) {$\scs i$};
    \node at (3,-.25) {$\scs n$};
    \filldraw  (1.5,.6) circle (2.5pt);
    \node at (0,1.75) {};
\end{tikzpicture}}
\right)
=
p!\hackcenter{\begin{tikzpicture}[scale=0.75]
    \draw[thick, ] (0,0) -- (0,1.5) ;
    \node at (.75,.75) {$ \cdots$};
    \draw[thick, ] (1.5,0) -- (1.5,1.5) ;
    \node at (2.25,.75) {$ \cdots$};
    \draw[thick, ] (3,0) -- (3,1.5) ;
    \node at (0,-.25) {$\scs 1$};
    \node at (1.5,-.25) {$\scs i$};
    \node at (3,-.25) {$\scs n$};
    \node at (2,.5) {$\scs p+1$};
    \filldraw  (1.5,.6) circle (2.5pt);
    \node at (0,1.75) {};
\end{tikzpicture}}
=0
$$
and
$$
\partial_{\mathrm{NH}}^2\left(
\hackcenter{\begin{tikzpicture}[scale=0.75]
    \draw[thick, ] (-.5,0) -- (-.5,1.5) ;
    \node at (.25,.75) {$\cdots$};
    \draw[thick, ] (1,0) .. controls (1,.75) and (1.5,.75) .. (1.5,1.5);
    \draw[thick, ] (1.5,0) .. controls (1.5,.75) and (1,.75) .. (1,1.5);
    \node at (2.25,.75) {$\cdots$};
    \draw[thick, ] (3,0) -- (3,1.5) ;
    \node at (-.5,-.25) {$\scs 1$};
    \node at (1,-.25) {$\scs i$};
    \node at (1.5,-.25) {$\scs i+1$};
    \node at (3,-.25) {$\scs n$};
    \node at (0,1.75) {};
\end{tikzpicture}}
\right)
=
2
\hackcenter{\begin{tikzpicture}[scale=0.75]
    \draw[thick, ] (-.5,0) -- (-.5,1.5) ;
    \node at (.25,.75) {$\cdots$};
    \draw[thick, ] (1,0) .. controls (1,.75) and (1.5,.75) .. (1.5,1.5);
    \draw[thick, ] (1.5,0) .. controls (1.5,.75) and (1,.75) .. (1,1.5)  node[pos=.2, shape=coordinate](DOT1){} node[pos=.8, shape=coordinate](DOT2){};
    \filldraw  (DOT1) circle (2.5pt);
    \filldraw  (DOT2) circle (2.5pt);
    \node at (2.25,.75) {$\cdots$};
    \draw[thick, ] (3,0) -- (3,1.5) ;
    \node at (-.5,-.25) {$\scs 1$};
    \node at (1,-.25) {$\scs i$};
    \node at (1.5,-.25) {$\scs i+1$};
    \node at (3,-.25) {$\scs n$};
    \node at (0,1.75) {};
\end{tikzpicture}}
\qquad
\qquad
\partial_{\mathrm{NH}}^3\left(
\hackcenter{\begin{tikzpicture}[scale=0.75]
    \draw[thick, ] (-.5,0) -- (-.5,1.5) ;
    \node at (.25,.75) {$\cdots$};
    \draw[thick, ] (1,0) .. controls (1,.75) and (1.5,.75) .. (1.5,1.5);
    \draw[thick, ] (1.5,0) .. controls (1.5,.75) and (1,.75) .. (1,1.5);
    \node at (2.25,.75) {$\cdots$};
    \draw[thick, ] (3,0) -- (3,1.5) ;
    \node at (-.5,-.25) {$\scs 1$};
    \node at (1,-.25) {$\scs i$};
    \node at (1.5,-.25) {$\scs i+1$};
    \node at (3,-.25) {$\scs n$};
    \node at (0,1.75) {};
\end{tikzpicture}}
\right)
=0.
$$

By Lemma \ref{lemma}, the differential $\partial_{\mathrm{NH}}$ is $p$-nilpotent.
Moreover, the differential $\partial_{\mathrm{NH}}$ preserves the nilHecke relations (this proof is left for the reader).
Hence, we have the following statement.
\begin{proposition}\label{nilhecke}
The nilHecke algebra with the $p$-derivation $(\mathrm{NH}_n,\partial_{\mathrm{NH}})$ is a $p$-DG algebra.
\end{proposition}

\subsection{The $p$-DG Webster algebra of type $A_1$}\label{pdg-webster}

Webster defined algebras which categorify tensor products of representations of quantum groups \cite{Web}.
Here we recall Webster algebra of type $A_1$ and $p$-DG structure of the Webster algebra\cite{KQYpDG}.

Let $m\geq 0$ be an integer.
For a sequence of non-negative integers $\mathbf{s}=(s_{1},...,s_{m})$, let $\mathrm{Seq}(\mathbf{s},n)$ be the set of all sequences $\mathbf{i}=(i_1,...,i_{m+n})$ in which $n$ of the entries are $\mf{b}$ and ${s}_{i}$ appears exactly once and in the order in which it appears in $\mathbf{s}$.
Denote by $\mathbf{i}_j$ the $j$-th entry of $\mathbf{i}$ and denote by $I_\mathbf{s}$ the set $\{{s}_{i}|1\leq i\leq m\}$.

Let $S_{m+n}$ be the symmetric group on $m+n$ letters generated by simple transpositions $\sigma_1, \ldots, \sigma_{m+n-1}$.  Each transposition $\sigma_j$ naturally acts on a sequence $\mathbf{i}$.  Note that if
$\mathbf{i} \in \mathrm{Seq}(\mathbf{s},n)$ it is not always the case that $\sigma_j.\mathbf{i} \in \mathrm{Seq}(\mathbf{s},n)$.

\begin{definition}
Webster algebra $W_n^{\mathbf{s}}$ of type $A_1$ is the $\Z$-graded $\Bbbk$-algebra generated by $e(\mathbf{i})$, where $\mathbf{i} \in \mathrm{Seq}(\mathbf{s},n)$, $x_j$, where $1\leq j\leq m+n$, and $\psi_\ell$, where $1\leq \ell \leq m+n-1$, satisfying the following relations.

\begin{minipage}{0.4\textwidth}
\begin{align}
& e(\mathbf{i})e(\mathbf{j}) = \delta_{\mathbf{i},\mathbf{j}}e(\mathbf{i})
\\
& x_j e(\mathbf{i})=e(\mathbf{i})x_j
\\
\label{dotonblack}
& x_j e(\mathbf{i} )=0 \quad \text{if $\mathbf{i}_j\in I_{\mathbf{s}}$}
\\
& \psi_j e(\mathbf{i}) = e(\sigma_j(\mathbf{i}))\psi_j
\end{align}
\end{minipage}
\begin{minipage}{0.6\textwidth}
\begin{align}
\label{red-red-crossing}
& \psi_j e(\mathbf{i}) = 0 \quad \text{if $\mathbf{i}_j,\mathbf{i}_{j+1}\in I_\mathbf{s}$}
\\
& x_j x_\ell=x_\ell x_j
\\
&\psi_j x_\ell = x_\ell\psi_j  \quad \text{if $\ell\not=j$ and $\ell\not=j+1$}
\\
& \psi_j\psi_\ell = \psi_\ell\psi_j \quad \text{if $|j - \ell| > 1$}
\end{align}
\end{minipage}
\begin{minipage}{\textwidth}
\begin{align}
&x_j \psi_j e(\mathbf{i}) - \psi_j x_{j+1} e(\mathbf{i})= \left\{
\begin{array}{ll}e(\mathbf{i}) & \text{$\mathbf{i}_j=\mathbf{i}_{j+1}=\mathbf{b}$}\\
0 & \text{$\mathbf{i}_j\in I_{\mathbf{s}},\, \mathbf{i}_{j+1}=\mathbf{b}$}
\end{array}
\right.
\\
&\psi_j x_j e(\mathbf{i}) - x_{j+1} \psi_j e(\mathbf{i})= \left\{
\begin{array}{ll}e(\mathbf{i}) & \text{$\mathbf{i}_j=\mathbf{i}_{j+1}=\mathbf{b}$}\\
0 & \text{$\mathbf{i}_j=\mathbf{b},\,\mathbf{i}_{j+1}\in I_{\mathbf{s}}$}
\end{array}
\right.
\\
&\psi_j^2 e(\mathbf{i})=\left\{
	\begin{array}{ll}
		0
		&\text{if } \mathbf{i}_j=\mathbf{i}_{j+1}={\mathfrak{b}},\\
		x_{j+1}^{\mathbf{i}_j} e(\mathbf{i})
		&\text{if } \mathbf{i}_j\in I_{\mathbf{s}}, \mathbf{i}_{j+1}={\mathfrak{b}},\\
		x_j^{\mathbf{i}_{j+1}}  e(\mathbf{i})
		&\text{if } \mathbf{i}_j={\mathfrak{b}}, \mathbf{i}_{j+1}\in I_{\mathbf{s}},
	\end{array}\right.
\\
& (\psi_j\psi_{j+1} \psi_j-\psi_{j+1}\psi_j \psi_{j+1}) e(\mathbf{i}) \\
& \qquad
=\left\{
	\begin{array}{ll}
	 \displaystyle \sum_{d_1+d_2=\mathbf{i}_{j+1}-1}x_j^{d_1}x_{j+2}^{d_2}e(\mathbf{i})
	&\text{if } \mathbf{i}_{j+1}\in I_{\mathbf{s}}, \mathbf{i}_{j}=\mathbf{i}_{j+2}={\mathfrak{b}},\\
	0
	&\text{otherwise},
	\end{array}\right.
	\\
\label{cyc}&e(\mathbf{i})=0 \qquad \text{if $\mathbf{i}_1=\mathfrak{b}$.}
\end{align}
\end{minipage}

The degrees of the generators are
\[
\deg x_i=2,\quad 
\deg \psi_j e(\mathbf{i})=
 \left\{
   \begin{array}{ll}
-2 & \hbox{if $\mathbf{i}_j=\mf{b}$, $\mathbf{i}_{j+1}=\mf{b}$,} \\
s & \hbox{if $\mathbf{i}_j=\mf{b}$, $\mathbf{i}_{j+1}=s\in I_{\mathbf{s}}$,} \\
s & \hbox{if $\mathbf{i}_j=s\in I_{\mathbf{s}}$, $\mathbf{i}_{j+1}=\mf{b}$.}
   \end{array}
 \right.
\]

\end{definition}

The generators $e(\mathbf{i})$ for $\mathbf{i} \in \mathrm{Seq}(\mathbf{s},n)$ represent the diagram consisting entirely of vertical strands whose $j$-th strand (counting from the left for each $j$) is black if $\mathbf{i}_j$ is $\mf{b}$ or thick red with $\mathbf{i}_j$-labeling if $\mathbf{i}_j$ is in $I_{\mathbf{s}}$.
Dots on black strands correspond to generators $x_j$.
A crossing of two strands corresponds to generators $\psi_j$.

The generators $x_j e(\mathbf{i})$ represents the dot on the $j$-th strand counting from the left of the $(m+n)$ strand diagram:
\begin{eqnarray*}
x_j e(\mathbf{i})&=&\hackcenter{\begin{tikzpicture}[scale=0.6]
    \draw[thick, ] (0,0) -- (0,1.5)  node[pos=.35, shape=coordinate](DOT){};
    \filldraw  (DOT) circle (2.5pt);
    \node at (-.75,.75) {$\cdots$};
    \node at (.75,.75) {$\cdots$};
\end{tikzpicture}}\quad  \hbox{if $\mathbf{i}_j=\mf{b}$.}
\end{eqnarray*}

The generator $\psi_j e(\mathbf{i})$ represents the crossing diagram as follows.
\[
\psi_j e(\mathbf{i})=
 \left\{
   \begin{array}{ll}
     \;
\hackcenter{\begin{tikzpicture}[scale=0.6]
    \draw[thick, ] (0,0) .. controls (0,.75) and (1.5,.75) .. (1.5,1.5);
    \draw[thick, ] (1.5,0) .. controls (1.5,.75) and (0,.75) .. (0,1.5);
    \node at (-.75,.75) {$\cdots$};
    \node at (2.25,.75) {$\cdots$};
\end{tikzpicture}} & \hbox{if $\mathbf{i}_j=\mf{b}$, $\mathbf{i}_{j+1}=\mf{b}$,} \\[2em]
\hackcenter{\begin{tikzpicture}[scale=0.6]
    \draw[thick, ] (0,0) .. controls (0,.75) and  (1.5,.75) .. (1.5,1.5);
    \draw[thick,red, double,  ] (1.5,0) .. controls (1.5,.75) and (0,.75) .. (0,1.5);
    \node at (1.5,-.25) {$\scs s$};
    \node at (-.75,.75) {$\cdots$};
    \node at (2.25,.75) {$\cdots$};
\end{tikzpicture}} & \hbox{if $\mathbf{i}_j=\mf{b}$, $\mathbf{i}_{j+1}=s\in I_{\mathbf{s}}$,} \\[2em]
\hackcenter{\begin{tikzpicture}[scale=0.6]
    \draw[thick,red, double,  ] (0,0) .. controls (0,.75) and  (1.5,.75) .. (1.5,1.5);
    \draw[thick, ] (1.5,0) .. controls (1.5,.75) and (0,.75) .. (0,1.5);
    \node at (0,-.25) {$\scs s$};
    \node at (-.75,.75) {$\cdots$};
    \node at (2.25,.75) {$\cdots$};
\end{tikzpicture}} & \hbox{if $\mathbf{i}_j=s\in I_{\mathbf{s}}$, $\mathbf{i}_{j+1}=\mf{b}$.}
   \end{array}
 \right.
\]

\begin{remark}
Remark that a black dot on any red strand and a red-red crossing do not appear by Relation \ref{dotonblack} and \ref{red-red-crossing}.
\end{remark}

The degrees of the generating diagrams are
\[
\deg
\left( \;
\hackcenter{\begin{tikzpicture}[scale=0.6]
    \draw[thick, ] (0,0) -- (0,1.5)  node[pos=.35, shape=coordinate](DOT){};
    \filldraw  (DOT) circle (2.5pt);
    \node at (0,-.25) {$\scs \;$};
    \node at (0,1.5) {$\scs \:$};
\end{tikzpicture}}\;
\right)
= 2,
\quad
\deg\left( \; \hackcenter{\begin{tikzpicture}[scale=0.6]
    \draw[thick, ] (0,0) .. controls (0,.75) and (.75,.75) .. (.75,1.5);
    \draw[thick, ] (.75,0) .. controls (.75,.75) and (0,.75) .. (0,1.5);
    \node at (0.5,-.25) {$\scs \;$};
    \node at (0.5,1.5) {$\scs \:$};
\end{tikzpicture}} \; \right) = -2,
\quad
\deg\left( \; \hackcenter{\begin{tikzpicture}[scale=0.6]
    \draw[thick, ] (0,0) .. controls (0,.75) and (.75,.75) .. (.75,1.5);
    \draw[thick,red, double,  ] (.75,0) .. controls (.75,.75) and (0,.75) .. (0,1.5);
    \node at (0.75,-.25) {$\scs s$};
    \node at (0,1.5) {$\scs \:$};
\end{tikzpicture}} \; \right) =
\deg\left( \; \hackcenter{\begin{tikzpicture}[scale=0.6]
    \draw[thick,red, double,  ] (0,0) .. controls (0,.75) and (.75,.75) .. (.75,1.5);
    \draw[thick, ] (.75,0) .. controls (.75,.75) and (0,.75) .. (0,1.5);
    \node at (0,-.25) {$\scs s$};
    \node at (0.75,1.5) {$\scs \:$};
\end{tikzpicture}} \; \right) = s.
\]

In the diagrammatic description, multiplication is vertical concatenation of diagrams.
The product is zero if the colors of the boundary points of one diagram do not match the colors of the boundary points of another diagram.
The diagrammatic description of the Webster algebra has local relations between generators which are the nilHecke relations among black strands
$$
\hackcenter{\begin{tikzpicture}[scale=0.75]
    \draw[thick] (0,0) .. controls ++(0,.375) and ++(0,-.375) .. (.75,.75);
    \draw[thick] (.75,0) .. controls ++(0,.375) and ++(0,-.375) .. (0,.75);
    \draw[thick] (0,.75 ) .. controls ++(0,.375) and ++(0,-.375) .. (.75,1.5);
    \draw[thick] (.75,.75) .. controls ++(0,.375) and ++(0,-.375) .. (0,1.5);
\end{tikzpicture}}
=0,
\qquad
\hackcenter{\begin{tikzpicture}[scale=0.75]
    \draw[thick,  ] (0,0) .. controls ++(0,1) and ++(0,-1) .. (1.2,2);
    \draw[thick] (.6,0) .. controls ++(0,.5) and ++(0,-.5) .. (0,1.0);
    \draw[thick] (0,1.0) .. controls ++(0,.5) and ++(0,-.5) .. (0.6,2);
    \draw[thick] (1.2,0) .. controls ++(0,1) and ++(0,-1) .. (0,2);
    \node at (0,2) {$\scs \:$};
\end{tikzpicture}}
\;\; = \;\;
\hackcenter{\begin{tikzpicture}[scale=0.75]
    \draw[thick] (0,0) .. controls ++(0,1) and ++(0,-1) .. (1.2,2);
    \draw[thick] (.6,0) .. controls ++(0,.5) and ++(0,-.5) .. (1.2,1.0);
    \draw[thick] (1.2,1.0) .. controls ++(0,.5) and ++(0,-.5) .. (0.6,2.0);
    \draw[thick] (1.2,0) .. controls ++(0,1) and ++(0,-1) .. (0,2.0);
    \node at (0,2) {$\scs \:$};
\end{tikzpicture}}
\quad,\qquad
\hackcenter{\begin{tikzpicture}[scale=0.75]
    \draw[thick] (0,0) .. controls (0,.75) and (.75,.75) .. (.75,1.5)
        node[pos=.25, shape=coordinate](DOT){};
    \draw[thick] (.75,0) .. controls (.75,.75) and (0,.75) .. (0,1.5);
    \filldraw  (DOT) circle (2.5pt);
\end{tikzpicture}}
\quad -\quad
\hackcenter{\begin{tikzpicture}[scale=0.75]
    \draw[thick] (0,0) .. controls (0,.75) and (.75,.75) .. (.75,1.5)
        node[pos=.75, shape=coordinate](DOT){};
    \draw[thick] (.75,0) .. controls (.75,.75) and (0,.75) .. (0,1.5);
    \filldraw  (DOT) circle (2.5pt);
\end{tikzpicture}}
\quad =\quad
\hackcenter{\begin{tikzpicture}[scale=0.75]
    \draw[thick] (0,0) -- (0,1.5);
    \draw[thick] (.75,0) -- (0.75,1.5);
\end{tikzpicture}}
\quad =\quad
\hackcenter{\begin{tikzpicture}[scale=0.75]
    \draw[thick] (0,0) .. controls (0,.75) and (.75,.75) .. (.75,1.5)
        node[pos=.75, shape=coordinate](DOT){};
    \draw[thick] (.75,0) .. controls (.75,.75) and (0,.75) .. (0,1.5);
    \filldraw  (DOT) circle (2.5pt);
\end{tikzpicture}}
\quad -\quad
\hackcenter{\begin{tikzpicture}[scale=0.75]
    \draw[thick] (0,0) .. controls (0,.75) and (.75,.75) .. (.75,1.5)
        node[pos=.25, shape=coordinate](DOT){};
    \draw[thick] (.75,0) .. controls (.75,.75) and (0,.75) .. (0,1.5);
    \filldraw  (DOT) circle (2.5pt);
\end{tikzpicture}}
$$
and local relations among red-black strands
\[
\hackcenter{\begin{tikzpicture}[scale=0.8]
    \draw[thick] (0,0) .. controls ++(0,.375) and ++(0,-.375) .. (.75,.75);
    \draw[thick,red, double, ] (.75,0) .. controls ++(0,.375) and ++(0,-.375) .. (0,.75);
    \draw[ thick,red, double, ] (0,.75 ) .. controls ++(0,.375) and ++(0,-.375) .. (.75,1.5);
    \draw[thick,  ] (.75,.75) .. controls ++(0,.375) and ++(0,-.375) .. (0,1.5);
    \node at (0.75,-.25) {$\scs s$};
    \node at (0,1.75) {$\scs \:$};
\end{tikzpicture}}
 \;\; = \;\;
\hackcenter{\begin{tikzpicture}[scale=0.8]
    \draw[thick, ] (0,0) -- (0,1.5)  node[pos=.55, shape=coordinate](DOT){};
    \filldraw  (DOT) circle (2.5pt);
    \draw[thick,red, double,  ] (.75,0) -- (.75,1.5);
    \node at (0.75,-.25) {$\scs s$};
    \node at (0,1.75) {$\scs \:$};
    \node at (-.25,.75) {$\scs s$};
\end{tikzpicture}}
\quad,\qquad
\hackcenter{\begin{tikzpicture}[scale=0.8]
    \draw[thick, red, double] (0,0) .. controls ++(0,.375) and ++(0,-.375) .. (.75,.75);
    \draw[thick ] (.75,0) .. controls ++(0,.375) and ++(0,-.375) .. (0,.75);
    \draw[ thick, ] (0,.75 ) .. controls ++(0,.375) and ++(0,-.375) .. (.75,1.5);
    \draw[thick,red, double,  ] (.75,.75) .. controls ++(0,.375) and ++(0,-.375) .. (0,1.5);
    \node at (0,-.25) {$\scs s$};
    \node at (0,1.75) {$\scs \:$};
\end{tikzpicture}}
 \;\; = \;\;
\hackcenter{\begin{tikzpicture}[scale=0.8]
    \draw[thick, ] (.75,0) -- (.75,1.5)  node[pos=.55, shape=coordinate](DOT){};
    \filldraw  (DOT) circle (2.5pt);
    \draw[thick,red, double,  ] (0,0) -- (0,1.5);
    \node at (0,-.25) {$\scs s$};
    \node at (0,1.75) {$\scs \:$};
    \node at (1,.75) {$\scs s$};
\end{tikzpicture}}
\quad,\qquad
\hackcenter{\begin{tikzpicture}[scale=0.8]
    \draw[thick, ] (0,0) .. controls (0,.75) and (.75,.75) .. (.75,1.5)
        node[pos=.25, shape=coordinate](DOT){};
    \draw[thick,red, double,  ] (.75,0) .. controls (.75,.75) and (0,.75) .. (0,1.5);
    \filldraw  (DOT) circle (2.5pt);
    \node at (0.75,-.25) {$\scs s$};
    \node at (0,1.75) {$\scs \:$};
\end{tikzpicture}}
\quad =\quad
\hackcenter{\begin{tikzpicture}[scale=0.8]
    \draw[thick, ] (0,0) .. controls (0,.75) and (.75,.75) .. (.75,1.5)
        node[pos=.75, shape=coordinate](DOT){};
    \draw[thick,red, double,  ] (.75,0) .. controls (.75,.75) and (0,.75) .. (0,1.5);
    \filldraw  (DOT) circle (2.5pt);
    \node at (0.75,-.25) {$\scs s$};
    \node at (0,1.75) {$\scs \:$};
\end{tikzpicture}}
\quad,\qquad
\hackcenter{\begin{tikzpicture}[scale=0.8]
    \draw[thick,red, double,  ] (0,0) .. controls (0,.75) and (.75,.75) .. (.75,1.5);
    \draw[thick,  ] (.75,0) .. controls (.75,.75) and (0,.75) .. (0,1.5)
        node[pos=.75, shape=coordinate](DOT){};
    \filldraw  (DOT) circle (2.75pt);
    \node at (0,-.25) {$\scs s$};
    \node at (0,1.75) {$\scs \:$};
\end{tikzpicture}}
\quad=\quad
\hackcenter{\begin{tikzpicture}[scale=0.8]
    \draw[thick,red, double,  ] (0,0) .. controls (0,.75) and (.75,.75) .. (.75,1.5);
    \draw[thick, ] (.75,0) .. controls (.75,.75) and (0,.75) .. (0,1.5)
        node[pos=.25, shape=coordinate](DOT){};
      \filldraw  (DOT) circle (2.75pt);
    \node at (0,-.25) {$\scs s$};
    \node at (0,1.75) {$\scs \:$};
\end{tikzpicture}}
\]

\[
\hackcenter{\begin{tikzpicture}[scale=0.8]
    \draw[thick,red, double,  ] (0,0) .. controls ++(0,1) and ++(0,-1) .. (1.2,2);
    \draw[thick, ] (.6,0) .. controls ++(0,.5) and ++(0,-.5) .. (0,1.0);
    \draw[thick, ] (0,1.0) .. controls ++(0,.5) and ++(0,-.5) .. (0.6,2);
    \draw[thick, ] (1.2,0) .. controls ++(0,1) and ++(0,-1) .. (0,2);
    \node at (0,-.25) {$\scs s$};
    \node at (0,2) {$\scs \:$};
\end{tikzpicture}}
\;\; = \;\;
\hackcenter{\begin{tikzpicture}[scale=0.8]
    \draw[thick,red, double,  ] (0,0) .. controls ++(0,1) and ++(0,-1) .. (1.2,2);
    \draw[thick, ] (.6,0) .. controls ++(0,.5) and ++(0,-.5) .. (1.2,1.0);
    \draw[thick, ] (1.2,1.0) .. controls ++(0,.5) and ++(0,-.5) .. (0.6,2.0);
    \draw[thick, ] (1.2,0) .. controls ++(0,1) and ++(0,-1) .. (0,2.0);
    \node at (0,-.25) {$\scs s$};
    \node at (0,2) {$\scs \:$};
\end{tikzpicture}}
\quad,\qquad
\hackcenter{\begin{tikzpicture}[scale=0.8]
    \draw[thick,  ] (0,0) .. controls ++(0,1) and ++(0,-1) .. (1.2,2);
    \draw[thick, ] (.6,0) .. controls ++(0,.5) and ++(0,-.5) .. (0,1.0);
    \draw[thick, ] (0,1.0) .. controls ++(0,.5) and ++(0,-.5) .. (0.6,2);
    \draw[thick,red, double,  ] (1.2,0) .. controls ++(0,1) and ++(0,-1) .. (0,2);
    \node at (1.2,-.25) {$\scs s$};
    \node at (0,2) {$\scs \:$};
\end{tikzpicture}}
\;\; = \;\;
\hackcenter{\begin{tikzpicture}[scale=0.8]
    \draw[thick,  ] (0,0) .. controls ++(0,1) and ++(0,-1) .. (1.2,2);
    \draw[thick, ] (.6,0) .. controls ++(0,.5) and ++(0,-.5) .. (1.2,1.0);
    \draw[thick, ] (1.2,1.0) .. controls ++(0,.5) and ++(0,-.5) .. (0.6,2.0);
    \draw[thick,red, double,  ] (1.2,0) .. controls ++(0,1) and ++(0,-1) .. (0,2.0);
    \node at (1.2,-.25) {$\scs s$};
    \node at (0,2) {$\scs \:$};
\end{tikzpicture}}
\]

\[
\hackcenter{\begin{tikzpicture}[scale=0.8]
    \draw[thick,  ] (0,0) .. controls ++(0,1) and ++(0,-1) .. (1.2,2);
    \draw[thick,red, double,  ] (.6,0) .. controls ++(0,.5) and ++(0,-.5) .. (0,1.0);
    \draw[thick,red, double,  ] (0,1.0) .. controls ++(0,.5) and ++(0,-.5) .. (0.6,2);
    \draw[thick, ] (1.2,0) .. controls ++(0,1) and ++(0,-1) .. (0,2);
    \node at (0.6,-.25) {$\scs s$};
    \node at (0,2) {$\scs \:$};
\end{tikzpicture}}
\;\; - \;\;
\hackcenter{\begin{tikzpicture}[scale=0.8]
    \draw[thick,  ] (0,0) .. controls ++(0,1) and ++(0,-1) .. (1.2,2);
    \draw[thick, red, double, ] (.6,0) .. controls ++(0,.5) and ++(0,-.5) .. (1.2,1.0);
    \draw[thick,red, double,  ] (1.2,1.0) .. controls ++(0,.5) and ++(0,-.5) .. (0.6,2.0);
    \draw[thick,  ] (1.2,0) .. controls ++(0,1) and ++(0,-1) .. (0,2.0);
    \node at (0.6,-.25) {$\scs s$};
    \node at (0,2) {$\scs \:$};
\end{tikzpicture}}
\;\; = \;\; 
\sum_{a+b=s-1}
\hackcenter{\begin{tikzpicture}[scale=0.8]
    \draw[thick, ] (0,0) -- (0,2)  node[pos=.5, shape=coordinate](DOT){};
    \draw[thick,red, double,  ] (.6,0) --  (.6,2);
    \draw[thick, ] (1.2,0) -- (1.2,2)  node[pos=.5, shape=coordinate](DOT1){};
    \node at (0.6,-.25) {$\scs s$};
    \node at (0,2) {$\scs \:$};
     \filldraw  (DOT) circle (2.75pt);
     \filldraw (DOT1) circle (2.75pt);
    \node at (-.25,1) {$\scs a$};
    \node at (1.45,1) {$\scs b$};
\end{tikzpicture}}
\]
and the cyclotomic relation \eqref{cyc} that a diagram equals to zero if a black strand appears at the most left strand of the diagram
\[
\hackcenter{\begin{tikzpicture}[scale=0.6]
    \draw[thick, ] (0,0) -- (0,1.5);
    \node at (.75,.75) {$\cdots$};
\end{tikzpicture}}
=0.
\]

The $p$-derivation $\partial_W$ on $W_n^{\mathbf{s}}$ is introduced in \cite{KQ}.

The Webster algebra $W^{\mathbf{s}}_n$ is a $p$-DG algebra with the following $p$-derivation for generators.

\begin{eqnarray}
&&\partial_W(x_i)=x_i^2,\\
&&\partial_W(\psi_j e(\mathbf{i}))=\left\{
\begin{array}{ll}
-x_j \psi_j e(\mathbf{i})- \psi_j x_{j+1}e(\mathbf{i}) &\text{if $\mathbf{i}_j=\mathbf{i}_{j+1}=\mathbf{b}$},\\
s \psi_j x_{j+1} e(\mathbf{i}) &\text{if $\mathbf{i}_j=s\in I_{\mathbf{s}},\mathbf{i}_{j+1}=\mathbf{b}$},\\
0 &\text{if $\mathbf{i}_j=\mathbf{b},\mathbf{i}_{j+1}=s\in I_{\mathbf{s}}$}.
\end{array}
\right.
\end{eqnarray}

By Leibniz rule, we have
\[
\partial_W^p(\psi_j e(\mathbf{i}))=
\prod_{a=0}^{p-1} (s+a) \psi_j x_{j+1}^p e(\mathbf{i})=0 \qquad\text{if $\mathbf{i}_j=s\in I_{\mathbf{s}},\mathbf{i}_{j+1}=\mathbf{b}$}.
\]

Therefore, the derivation $\partial_W$ is $p$-nilpotent for generators.
Using Lemma \ref{lemma}, we have $\partial_W^p=0$ on the Webster algebra.

Moreover, we find that the derivation $\partial_W$ preserves the relations of the Webster algebra.
Hence, we have the following statement.
\begin{proposition}
The Webster algebra of type $A_1$ with the $p$-derivation $(W_n^{\mathbf{s}},\partial_W)$ is a $p$-DG algebra.
\end{proposition}

\section{Deformed Webster algebra $D({\bf s},n)$ and splitter bimodules}
A deformation of non-cyclotomic Webster algebras of type $A_1$ for ${\bf s}=(1^m)$ is introduced by Khovanov and Sussan\cite{KS} and the deformed Webster algebra is generalized for the general ${\bf s}$ in \cite{KLSY}.
In this section, we recall the deformed Webster algebras $D({\bf s},n)$, their splitter bimodules, and these diagrammatic description defined in \cite{KLSY}.
In the next section, we will define a $p$-DG structure on the deformed Webster algebra of type $A_1$ and splitter bimodules.
\subsection{Definition of $D({\bf s},n)$}
Let $m\geq 0$ be an integer.
For a sequence of non-negative integers $\mathbf{s}=(s_{1},...,s_{m})$, let $\mathrm{Seq}(\mathbf{s},n)$ be the set of all sequences $\mathbf{i}=(i_1,...,i_{m+n})$ in which $n$ of the entries are $\mf{b}$ and ${s}_{i}$ appears exactly once and in the order in which it appears in $\mathbf{s}$.
Denote by $\mathbf{i}_j$ the $j$-th entry of $\mathbf{i}$ and denote by $I_\mathbf{s}$ the set $\{{s}_{i}|1\leq i\leq m\}$.

Let $S_{m+n}$ be the symmetric group on $m+n$ letters generated by simple transpositions $\sigma_1, \ldots, \sigma_{m+n-1}$.  Each transposition $\sigma_j$ naturally acts on a sequence $\mathbf{i}$.  Note that if
$\mathbf{i} \in \mathrm{Seq}(\mathbf{s},n)$ it is not always the case that $\sigma_j.\mathbf{i} \in \mathrm{Seq}(\mathbf{s},n)$.

Let $\Bbbk$ be a field of characteristic $p > 0$.
$D(\mathbf{s},n)$ is the $\Z$-graded $\Bbbk$-algebra generated by $e(\mathbf{i})$, where $\mathbf{i} \in \mathrm{Seq}(\mathbf{s},n)$, $x_j$, $E(d)_j$, where $1\leq j\leq m+n$, $d\geq 1$, and $\psi_\ell$, where $1\leq \ell \leq m+n-1$, satisfying the relations below. For convenience, we use the notation $E(0)_j=1$.

\begin{minipage}{0.4\textwidth}
\begin{align}
& e(\mathbf{i})e(\mathbf{j}) = \delta_{\mathbf{i},\mathbf{j}}e(\mathbf{i})
\\
& E(d)_j e(\mathbf{i})=e(\mathbf{i})E(d)_j
\\
& E(d)_j e(\mathbf{i})=0 \quad \text{if $\mathbf{i}_j=\mathfrak{b}$}
\\
&E(d)_j e({\bf i})=0 \quad \text{if $d > {\bf i}_j \in I_{\mathbf{s}}$}
\\
& x_j e(\mathbf{i})=e(\mathbf{i})x_j
\\
& x_j e(\mathbf{i} )=0 \quad \text{if $\mathbf{i}_j\in I_{\mathbf{s}}$}
\\
& \psi_j e(\mathbf{i}) = e(\sigma_j(\mathbf{i}))\psi_j
\\
& \psi_j e(\mathbf{i}) = 0 \quad \text{if $\mathbf{i}_j,\mathbf{i}_{j+1}\in I_\mathbf{s}$}
\end{align}
\end{minipage}
\begin{minipage}{0.6\textwidth}
\begin{align}
& x_j x_\ell=x_\ell x_j
\\
& E(d)_j x_\ell=x_\ell E(d)_j
\\
& E(d)_j E(d')_\ell=E(d')_\ell E(d)_j
\\
&\psi_j x_\ell = x_\ell\psi_j  \quad \text{if $\ell\not=j$ and $\ell\not=j+1$}
\\
&\psi_j E(d)_\ell = E(d)_\ell\psi_j  \quad \text{if $\ell\not=j$ and $\ell\not=j+1$}
\\
& \psi_j E(d)_j  = E(d)_{j+1} \psi_j
\\
& E(d)_j \psi_j = \psi_j E(d)_{j+1}
\\
& \psi_j\psi_\ell = \psi_\ell\psi_j \quad \text{if $|j - \ell| > 1$}
\end{align}
\end{minipage}
\begin{minipage}{\textwidth}
\begin{align}
\label{nil-rel1}
&x_j \psi_j e(\mathbf{i}) - \psi_j x_{j+1} e(\mathbf{i})= \left\{
\begin{array}{ll}e(\mathbf{i}) & \text{$\mathbf{i}_j=\mathbf{i}_{j+1}=\mathbf{b}$}\\
0 & \text{$\mathbf{i}_j\in I_{\mathbf{s}},\, \mathbf{i}_{j+1}=\mathbf{b}$}
\end{array}
\right.
\\
&\psi_j x_j e(\mathbf{i}) - x_{j+1} \psi_j e(\mathbf{i})= \left\{
\begin{array}{ll}e(\mathbf{i}) & \text{$\mathbf{i}_j=\mathbf{i}_{j+1}=\mathbf{b}$}\\
0 & \text{$\mathbf{i}_j=\mathbf{b},\,\mathbf{i}_{j+1}\in I_{\mathbf{s}}$}
\end{array}
\right.
\\
\label{rel:r2}
&\psi_j^2 e(\mathbf{i})=\left\{
	\begin{array}{ll}
		0
		&\text{if } \mathbf{i}_j=\mathbf{i}_{j+1}={\mathfrak{b}},\\
		\displaystyle \sum_{d=0}^{\mathbf{i}_j}(-1)^{d}E(d)_j x_{j+1}^{\mathbf{i}_j-d} e(\mathbf{i})
		&\text{if } \mathbf{i}_j\in I_{\mathbf{s}}, \mathbf{i}_{j+1}={\mathfrak{b}},\\
		\displaystyle \sum_{d=0}^{\mathbf{i}_{j+1}}(-1)^{d} x_j^{\mathbf{i}_{j+1}-d} E(d)_{j+1} e(\mathbf{i})
		&\text{if } \mathbf{i}_j={\mathfrak{b}}, \mathbf{i}_{j+1}\in I_{\mathbf{s}},
	\end{array}\right.
\\
\label{rel:r3}
& (\psi_j\psi_{j+1} \psi_j-\psi_{j+1}\psi_j \psi_{j+1}) e(\mathbf{i}) \\
\nonumber
& \qquad
=\left\{
	\begin{array}{ll}
	 \displaystyle \sum_{d_1+d_2+d_3=\mathbf{i}_{j+1}-1}(-1)^{d_3}x_j^{d_1}E(d_3)_{j+1}x_{j+2}^{d_2}e(\mathbf{i})
	&\text{if } \mathbf{i}_{j+1}\in I_{\mathbf{s}}, \mathbf{i}_{j}=\mathbf{i}_{j+2}={\mathfrak{b}},\\
	0
	&\text{otherwise},
	\end{array}\right.
\end{align}
\end{minipage}

The degrees of the generators are
\[
\deg x_i=2,\quad 
\deg E(d)_j=2d,\quad 
\deg \psi_j e(\mathbf{i})=
 \left\{
   \begin{array}{ll}
-2 & \hbox{if $\mathbf{i}_j=\mf{b}$, $\mathbf{i}_{j+1}=\mf{b}$,} \\
s & \hbox{if $\mathbf{i}_j=\mf{b}$, $\mathbf{i}_{j+1}=s\in I_{\mathbf{s}}$,} \\
s & \hbox{if $\mathbf{i}_j=s\in I_{\mathbf{s}}$, $\mathbf{i}_{j+1}=\mf{b}$.}
   \end{array}
 \right.
\]


\subsection{Inclusion map}

For a sequence of $m$ non-negative integers $\mathbf{s} = (s_{1}, s_{2}, \dots , s_{m})$, we define the sequence $\mathbf{s}^j$, where $j=1,\ldots, m-1$, by the sequence of $m-1$ integers obtained from $\mathbf{s}$ by replacing the pair $(s_{j}, s_{j+1})$ with the singleton $s_{j}+s_{j+1}$ and we define the sequence $\phi_{j,a}(\mathbf{s})$, where $j=1,...,m$ and $0\leq a\leq s_{j}$, by the sequence of $m+1$ integers obtained from $\mathbf{s}$ by replacing the integer $s_j$ with the pair $(s_j-a,a)$:
\begin{eqnarray*}
\mathbf{s}^{j} &=& (s_{1}, \dots, s_{j-1},  s_{j}+  s_{j+1}, s_{j+2}, \dots, s_{m}),\\
\phi_{j,a}(\mathbf{s})&=&(s_1,\ldots,s_{j-1},s_{j}-a,a,s_{j+1},\ldots, s_m).
\end{eqnarray*}
Note that we have $\phi_{j,s_{j+1}}(\mathbf{s}^j)=\mathbf{s}$.

The map $\phi_{j,a}$ extends to a map from $\mathrm{Seq}(\mathbf{s},n)$ to $\mathrm{Seq}(\phi_{j,a}(\mathbf{s}),n)$ by replacing the integer $s_j$ in $\mathbf{i}$ with the pair $(s_j-a,a)$.
When $\mathbf{i}_\ell=s_j$,
\begin{eqnarray*}
\phi_{j,a}(\mathbf{i})&=& (...\mathbf{i}_{\ell-1},s_{j}-a,a,\mathbf{i}_{\ell+1},...)\in \mathrm{Seq}(\phi_{j,a}(\mathbf{s}),n).
\end{eqnarray*}

We define an inclusion map of algebras
\begin{equation}
 \Phi_{j,a} : D(\mathbf{s},n) \to D(\phi_{j,a}(\mathbf{s}), n)
\end{equation}
determined by sending idempotents $e(\mathbf{i})$ for $\mathrm{Seq}(\mathbf{s},n)$ by
\begin{equation*}
\Phi_{j,a}(e(\mathbf{i})) = e(\phi_{j,a}(\mathbf{i}))
\end{equation*}

For generators $E(d)_\ell$, $\Phi_{j,a}$ is defined by
\begin{equation*}
\Phi_{j,a}(E(d)_\ell e(\mathbf{i})) =
\begin{cases}
E(d)_\ell e(\phi_{j,a}(\mathbf{i}))& \text{ if } \mathbf{i}_{\ell'}=s_j, \ell<\ell' \\
E(d)_{\ell+1} e(\phi_{j,a}(\mathbf{i}))& \text{ if } \mathbf{i}_{\ell'}=s_j, \ell'<\ell \\
\sum_{d_1+d_2=d}E(d_1)_{\ell}E(d_2)_{\ell+1} e(\phi_{j,a}(\mathbf{i}))& \text{ if } \mathbf{i}_{\ell'}=s_j, \ell'=\ell, 1\leq d\leq s_j
\end{cases}
\end{equation*}

For generators $x_\ell$, $\Phi_{j,a}$ is defined by
\begin{equation*}
\Phi_{j,a}(x_\ell e(\mathbf{i})) =
\begin{cases}
x_\ell e(\phi_{j,a}(\mathbf{i}))& \text{ if } \mathbf{i}_{\ell'}=s_j, \ell<\ell' \\
x_{\ell+1} e(\phi_{j,a}(\mathbf{i}))& \text{ if } \mathbf{i}_{\ell'}=s_j, \ell'<\ell
\end{cases}
\end{equation*}

For generators $\psi_{\ell}$, $\Phi_{j,a}$ is defined by
\begin{equation*}
\Phi_{j,a}(\psi_\ell e(\mathbf{i})) =
\begin{cases}
\psi_{\ell} e(\phi_{j,a}(\mathbf{i})) & \text{ if } \mathbf{i}_{\ell'}=s_j, \ell< \ell'-1 \\
\psi_{\ell+1}\psi_{\ell} e(\phi_{j,a}(\mathbf{i})) & \text{ if } \mathbf{i}_{\ell'}=s_j,\mathbf{i}_{\ell'-1}=\mf{b}, \ell= \ell'-1 \\
\psi_{\ell}\psi_{\ell+1} e(\phi_{j,a}(\mathbf{i})) & \text{ if } \mathbf{i}_{\ell'}=s_j,\mathbf{i}_{\ell'+1}=\mf{b}, \ell= \ell' \\
\psi_{\ell+1} e(\phi_{j,a}(\mathbf{i})) & \text{ if } \mathbf{i}_{\ell'}=s_j, \ell> \ell'
\end{cases}
\end{equation*}

\subsection{Splitter bimodules}
Let $\mathrm{Seq}^j(\mathbf{s},n)$ be a subset of $\mathrm{Seq}(\mathbf{s},n)$ composed of sequences $\mathbf{i}=(i_1,...,i_{m+n})$ in which $s_j$ is neighbor to $s_{j+1}$, that is there is $1\leq k\leq m+n-1$ such that $\mathbf{i}_k=s_j$ and $\mathbf{i}_{k+1}=s_{j+1}$, and put
\[
\hat{e_j}=\sum_{\mathbf{i}\in \mathrm{Seq}^j(\mathbf{s},n)}e(\mathbf{i}).
\]

First note that the inclusion $\Phi_{j,s_{j+1}}$ determines the left action of $D(\mathbf{s}^j,n)$ on $\hat{e_j}D(\mathbf{s},n)$ and the right action on $D(\mathbf{s},n)\hat{e_j}$.
Define the $(D(\mathbf{s}^j,n) , D(\mathbf{s},n))$-bimodule $\bigtriangleup^j(\mathbf{s})$ by
\begin{equation}
 \bigtriangleup^j(\mathbf{s}) := D(\mathbf{s}^j,n) \otimes_{D(\mathbf{s}^j,n)} \hat{e_j}D(\mathbf{s},n)\{-s_j \cdot s_{j+1}\}
\end{equation}
and the $(D(\mathbf{s},n),D(\mathbf{s}^j,n))$-bimodule $\bigtriangledown^j(\mathbf{s})$ by
\begin{equation}
 \bigtriangledown^j(\mathbf{s}) := D(\mathbf{s},n)\hat{e_j} \otimes_{D(\mathbf{s}^j,n)} D(\mathbf{s}^j,n).
\end{equation}

\subsection{Diagrammatic description}

The generators $e(\mathbf{i})$ for $\mathbf{i} \in \mathrm{Seq}(\mathbf{s},n)$ represent the diagram consisting entirely of vertical strands whose $j$-th strand (counting from the left for each $j$) is black if $\mathbf{i}_j$ is $\mf{b}$ or thick red with $\mathbf{i}_j$-labeling if $\mathbf{i}_j$ is in $I_{\mathbf{s}}$.

The generators $x_j e(\mathbf{i})$ and $E(d)_j e(\mathbf{i})$ represent the dot on the $j$-th strand counting from the left of the $(m+n)$ strand diagram:
\begin{eqnarray*}
x_j e(\mathbf{i})&=&\hackcenter{\begin{tikzpicture}[scale=0.6]
    \draw[thick, ] (0,0) -- (0,1.5)  node[pos=.35, shape=coordinate](DOT){};
    \filldraw  (DOT) circle (2.5pt);
\end{tikzpicture}}\quad  \hbox{if $\mathbf{i}_j=\mf{b}$,} \\
E(d)_j e(\mathbf{i})&=&
\hackcenter{\begin{tikzpicture}[scale=0.6]
    \draw[thick,red, double,  ] (0,0) to (0,1.5);
 \filldraw[red]  (0,.5) circle (2.75pt);
    \node at (0,-.25) {$\scs s$};
    \node at (-.5,.5) {$\scs E_d$};
\end{tikzpicture}}\quad  \hbox{if $\mathbf{i}_j=s$ and $d\leq s$.}
\end{eqnarray*}

The generator $\psi_j e(\mathbf{i})$ represents the crossing diagram of $(j,j+1)$ strands as follows.
\[
\psi_j e(\mathbf{i})=
 \left\{
   \begin{array}{ll}
     \;
\hackcenter{\begin{tikzpicture}[scale=0.6]
    \draw[thick, ] (0,0) .. controls (0,.75) and (1.5,.75) .. (1.5,1.5);
    \draw[thick, ] (1.5,0) .. controls (1.5,.75) and (0,.75) .. (0,1.5);
\end{tikzpicture}}, & \hbox{if $\mathbf{i}_j=\mf{b}$, $\mathbf{i}_{j+1}=\mf{b}$;} \\[2em]
\hackcenter{\begin{tikzpicture}[scale=0.6]
    \draw[thick, ] (0,0) .. controls (0,.75) and  (1.5,.75) .. (1.5,1.5);
    \draw[thick,red, double,  ] (1.5,0) .. controls (1.5,.75) and (0,.75) .. (0,1.5);
    \node at (1.5,-.25) {$\scs s$};
\end{tikzpicture}}, & \hbox{if $\mathbf{i}_j=\mf{b}$, $\mathbf{i}_{j+1}=s\in I_{\mathbf{s}}$;} \\[2em]
\hackcenter{\begin{tikzpicture}[scale=0.6]
    \draw[thick,red, double,  ] (0,0) .. controls (0,.75) and  (1.5,.75) .. (1.5,1.5);
    \draw[thick, ] (1.5,0) .. controls (1.5,.75) and (0,.75) .. (0,1.5);
    \node at (0,-.25) {$\scs s$};
\end{tikzpicture}}, & \hbox{if $\mathbf{i}_j=s\in I_{\mathbf{s}}$, $\mathbf{i}_{j+1}=\mf{b}$.}
   \end{array}
 \right.
\]

The degrees of the generating diagrams are
\[
\deg
\left( \;
\hackcenter{\begin{tikzpicture}[scale=0.6]
    \draw[thick, ] (0,0) -- (0,1.5)  node[pos=.35, shape=coordinate](DOT){};
    \filldraw  (DOT) circle (2.5pt);
    \node at (0,-.25) {$\scs \;$};
    \node at (0,1.5) {$\scs \:$};
\end{tikzpicture}}\;
\right)
= 2,
\quad
\deg
\left( \;
\hackcenter{\begin{tikzpicture}[scale=0.6]
    \draw[thick,red, double,  ] (0,0) to (0,1.5);
 \filldraw[red]  (0,.5) circle (2.75pt);
    \node at (0,-.25) {$\scs s$};
    \node at (-.5,.5) {$\scs E_d$};
    \node at (0,1.5) {$\scs \:$};
\end{tikzpicture}}\;
\right)
= 2d,
\quad
\deg\left( \; \hackcenter{\begin{tikzpicture}[scale=0.6]
    \draw[thick, ] (0,0) .. controls (0,.75) and (.75,.75) .. (.75,1.5);
    \draw[thick, ] (.75,0) .. controls (.75,.75) and (0,.75) .. (0,1.5);
    \node at (0.5,-.25) {$\scs \;$};
    \node at (0.5,1.5) {$\scs \:$};
\end{tikzpicture}} \; \right) = -2,
\quad
\deg\left( \; \hackcenter{\begin{tikzpicture}[scale=0.6]
    \draw[thick, ] (0,0) .. controls (0,.75) and (.75,.75) .. (.75,1.5);
    \draw[thick,red, double,  ] (.75,0) .. controls (.75,.75) and (0,.75) .. (0,1.5);
    \node at (0.75,-.25) {$\scs s$};
    \node at (0,1.5) {$\scs \:$};
\end{tikzpicture}} \; \right) =
\deg\left( \; \hackcenter{\begin{tikzpicture}[scale=0.6]
    \draw[thick,red, double,  ] (0,0) .. controls (0,.75) and (.75,.75) .. (.75,1.5);
    \draw[thick, ] (.75,0) .. controls (.75,.75) and (0,.75) .. (0,1.5);
    \node at (0,-.25) {$\scs s$};
    \node at (0.75,1.5) {$\scs \:$};
\end{tikzpicture}} \; \right) = s.
\]

The elements $D(\mathbf{s},n)$ are formal $\Bbbk$-linear combinations of these diagrams modulo isotopy, nilHecke relations, and the following local relations.
We give the algebra structure of $D(\mathbf{s},n)$ by concatenating diagrams vertically when the colors on the endpoints of two diagrams match.

\begin{eqnarray}
&&
\label{RBr2-rel}
\hackcenter{\begin{tikzpicture}[scale=0.8]
    \draw[thick] (0,0) .. controls ++(0,.375) and ++(0,-.375) .. (.75,.75);
    \draw[thick,red, double, ] (.75,0) .. controls ++(0,.375) and ++(0,-.375) .. (0,.75);
    \draw[ thick,red, double, ] (0,.75 ) .. controls ++(0,.375) and ++(0,-.375) .. (.75,1.5);
    \draw[thick,  ] (.75,.75) .. controls ++(0,.375) and ++(0,-.375) .. (0,1.5);
    \node at (0.75,-.25) {$\scs s$};
    \node at (0,1.75) {$\scs \:$};
\end{tikzpicture}}
 \;\; = \;\;
 \sum_{k=0}^s (-1)^k
\hackcenter{\begin{tikzpicture}[scale=0.8]
    \draw[thick, ] (0,0) -- (0,1.5)  node[pos=.55, shape=coordinate](DOT){};
    \filldraw  (DOT) circle (2.5pt);
    \draw[thick,red, double,  ] (.75,0) -- (.75,1.5) node[pos=.55, shape=coordinate](DOT2){};
    \node at (0.75,-.25) {$\scs s$};
    \node at (0,1.75) {$\scs \:$};
    \filldraw[red]  (DOT2) circle (2.75pt);
    \node at (-.75,.75) {$\scs s-k$};
    \node at (1.25,.75) {$\scs E_k$};
\end{tikzpicture}}
\qquad \quad
\hackcenter{\begin{tikzpicture}[scale=0.8]
    \draw[thick, red, double] (0,0) .. controls ++(0,.375) and ++(0,-.375) .. (.75,.75);
    \draw[thick ] (.75,0) .. controls ++(0,.375) and ++(0,-.375) .. (0,.75);
    \draw[ thick, ] (0,.75 ) .. controls ++(0,.375) and ++(0,-.375) .. (.75,1.5);
    \draw[thick,red, double,  ] (.75,.75) .. controls ++(0,.375) and ++(0,-.375) .. (0,1.5);
    \node at (0,-.25) {$\scs s$};
    \node at (0,1.75) {$\scs \:$};
\end{tikzpicture}}
 \;\; = \;\;
 \sum_{k=0}^s (-1)^k
\hackcenter{\begin{tikzpicture}[scale=0.8]
    \draw[thick, ] (.75,0) -- (.75,1.5)  node[pos=.55, shape=coordinate](DOT){};
    \filldraw  (DOT) circle (2.5pt);
    \draw[thick,red, double,  ] (0,0) -- (0,1.5) node[pos=.55, shape=coordinate](DOT2){};
    \filldraw[red]  (DOT2) circle (2.75pt);
    \node at (0,-.25) {$\scs s$};
    \node at (0,1.75) {$\scs \:$};
   \node at (-.5,.75) {$\scs E_k$};
    \node at (1.5,.75) {$\scs s-k$};
\end{tikzpicture}}
\end{eqnarray}
where we set $E_0=1$.
\begin{eqnarray}
&&
\label{blackdot}
\hackcenter{\begin{tikzpicture}[scale=0.8]
    \draw[thick, ] (0,0) .. controls (0,.75) and (.75,.75) .. (.75,1.5)
        node[pos=.25, shape=coordinate](DOT){};
    \draw[thick,red, double,  ] (.75,0) .. controls (.75,.75) and (0,.75) .. (0,1.5);
    \filldraw  (DOT) circle (2.5pt);
    \node at (0.75,-.25) {$\scs s$};
    \node at (0,1.75) {$\scs \:$};
\end{tikzpicture}}
\quad =\quad
\hackcenter{\begin{tikzpicture}[scale=0.8]
    \draw[thick, ] (0,0) .. controls (0,.75) and (.75,.75) .. (.75,1.5)
        node[pos=.75, shape=coordinate](DOT){};
    \draw[thick,red, double,  ] (.75,0) .. controls (.75,.75) and (0,.75) .. (0,1.5);
    \filldraw  (DOT) circle (2.5pt);
    \node at (0.75,-.25) {$\scs s$};
    \node at (0,1.75) {$\scs \:$};
\end{tikzpicture}}
\qquad \quad
\hackcenter{\begin{tikzpicture}[scale=0.8]
    \draw[thick,red, double,  ] (0,0) .. controls (0,.75) and (.75,.75) .. (.75,1.5);
    \draw[thick,  ] (.75,0) .. controls (.75,.75) and (0,.75) .. (0,1.5)
        node[pos=.75, shape=coordinate](DOT){};
    \filldraw  (DOT) circle (2.75pt);
    \node at (0,-.25) {$\scs s$};
    \node at (0,1.75) {$\scs \:$};
\end{tikzpicture}}
\quad=\quad
\hackcenter{\begin{tikzpicture}[scale=0.8]
    \draw[thick,red, double,  ] (0,0) .. controls (0,.75) and (.75,.75) .. (.75,1.5);
    \draw[thick, ] (.75,0) .. controls (.75,.75) and (0,.75) .. (0,1.5)
        node[pos=.25, shape=coordinate](DOT){};
      \filldraw  (DOT) circle (2.75pt);
    \node at (0,-.25) {$\scs s$};
    \node at (0,1.75) {$\scs \:$};
\end{tikzpicture}}
\\
\label{reddot}
&&
\hackcenter{\begin{tikzpicture}[scale=0.8]
    \draw[thick,red, double,  ] (0,0) .. controls (0,.75) and (.75,.75) .. (.75,1.5)
        node[pos=.25, shape=coordinate](DOT){};
    \draw[thick, ] (.75,0) .. controls (.75,.75) and (0,.75) .. (0,1.5);
    \filldraw[red]  (DOT) circle (2.5pt);
    \node at (0,-.25) {$\scs s$};
    \node at (0,1.75) {$\scs \:$};
    \node at (-.3,.4) {$\scs E_d$};
\end{tikzpicture}}
\quad =\quad
\hackcenter{\begin{tikzpicture}[scale=0.8]
    \draw[thick,red, double,  ] (0,0) .. controls (0,.75) and (.75,.75) .. (.75,1.5)
        node[pos=.75, shape=coordinate](DOT){};
    \draw[thick, ] (.75,0) .. controls (.75,.75) and (0,.75) .. (0,1.5);
    \filldraw[red]  (DOT) circle (2.5pt);
    \node at (0,-.25) {$\scs s$};
    \node at (0,1.75) {$\scs \:$};
    \node at (1.05,1.1) {$\scs E_d$};
\end{tikzpicture}}
\qquad \quad
\hackcenter{\begin{tikzpicture}[scale=0.8]
    \draw[thick, ] (0,0) .. controls (0,.75) and (.75,.75) .. (.75,1.5);
    \draw[thick,red, double,  ] (.75,0) .. controls (.75,.75) and (0,.75) .. (0,1.5)
        node[pos=.75, shape=coordinate](DOT){};
    \filldraw[red]  (DOT) circle (2.75pt);
    \node at (0.75,-.25) {$\scs s$};
    \node at (0,1.75) {$\scs \:$};
    \node at (-.3,1.1) {$\scs E_d$};
\end{tikzpicture}}
\quad=\quad
\hackcenter{\begin{tikzpicture}[scale=0.8]
    \draw[thick, ] (0,0) .. controls (0,.75) and (.75,.75) .. (.75,1.5);
    \draw[thick,red, double,  ] (.75,0) .. controls (.75,.75) and (0,.75) .. (0,1.5)
        node[pos=.25, shape=coordinate](DOT){};
      \filldraw[red]  (DOT) circle (2.75pt);
    \node at (0.75,-.25) {$\scs s$};
    \node at (0,1.75) {$\scs \:$};
    \node at (1.05,.4) {$\scs E_d$};
\end{tikzpicture}}
\end{eqnarray}

\begin{align}
\hackcenter{\begin{tikzpicture}[scale=0.8]
    \draw[thick,red, double,  ] (0,0) .. controls ++(0,1) and ++(0,-1) .. (1.2,2);
    \draw[thick, ] (.6,0) .. controls ++(0,.5) and ++(0,-.5) .. (0,1.0);
    \draw[thick, ] (0,1.0) .. controls ++(0,.5) and ++(0,-.5) .. (0.6,2);
    \draw[thick, ] (1.2,0) .. controls ++(0,1) and ++(0,-1) .. (0,2);
    \node at (0,-.25) {$\scs s$};
    \node at (0,2) {$\scs \:$};
\end{tikzpicture}}
\;\; = \;\;
\hackcenter{\begin{tikzpicture}[scale=0.8]
    \draw[thick,red, double,  ] (0,0) .. controls ++(0,1) and ++(0,-1) .. (1.2,2);
    \draw[thick, ] (.6,0) .. controls ++(0,.5) and ++(0,-.5) .. (1.2,1.0);
    \draw[thick, ] (1.2,1.0) .. controls ++(0,.5) and ++(0,-.5) .. (0.6,2.0);
    \draw[thick, ] (1.2,0) .. controls ++(0,1) and ++(0,-1) .. (0,2.0);
    \node at (0,-.25) {$\scs s$};
    \node at (0,2) {$\scs \:$};
\end{tikzpicture}}
\qquad \qquad
\hackcenter{\begin{tikzpicture}[scale=0.8]
    \draw[thick,  ] (0,0) .. controls ++(0,1) and ++(0,-1) .. (1.2,2);
    \draw[thick, ] (.6,0) .. controls ++(0,.5) and ++(0,-.5) .. (0,1.0);
    \draw[thick, ] (0,1.0) .. controls ++(0,.5) and ++(0,-.5) .. (0.6,2);
    \draw[thick,red, double,  ] (1.2,0) .. controls ++(0,1) and ++(0,-1) .. (0,2);
    \node at (1.2,-.25) {$\scs s$};
    \node at (0,2) {$\scs \:$};
\end{tikzpicture}}
\;\; = \;\;
\hackcenter{\begin{tikzpicture}[scale=0.8]
    \draw[thick,  ] (0,0) .. controls ++(0,1) and ++(0,-1) .. (1.2,2);
    \draw[thick, ] (.6,0) .. controls ++(0,.5) and ++(0,-.5) .. (1.2,1.0);
    \draw[thick, ] (1.2,1.0) .. controls ++(0,.5) and ++(0,-.5) .. (0.6,2.0);
    \draw[thick,red, double,  ] (1.2,0) .. controls ++(0,1) and ++(0,-1) .. (0,2.0);
    \node at (1.2,-.25) {$\scs s$};
    \node at (0,2) {$\scs \:$};
\end{tikzpicture}}
\end{align}

\begin{equation}\label{r3-2}
\hackcenter{\begin{tikzpicture}[scale=0.8]
    \draw[thick,  ] (0,0) .. controls ++(0,1) and ++(0,-1) .. (1.2,2);
    \draw[thick,red, double,  ] (.6,0) .. controls ++(0,.5) and ++(0,-.5) .. (0,1.0);
    \draw[thick,red, double,  ] (0,1.0) .. controls ++(0,.5) and ++(0,-.5) .. (0.6,2);
    \draw[thick, ] (1.2,0) .. controls ++(0,1) and ++(0,-1) .. (0,2);
    \node at (0.6,-.25) {$\scs s$};
    \node at (0,2) {$\scs \:$};
\end{tikzpicture}}
\;\; - \;\;
\hackcenter{\begin{tikzpicture}[scale=0.8]
    \draw[thick,  ] (0,0) .. controls ++(0,1) and ++(0,-1) .. (1.2,2);
    \draw[thick, red, double, ] (.6,0) .. controls ++(0,.5) and ++(0,-.5) .. (1.2,1.0);
    \draw[thick,red, double,  ] (1.2,1.0) .. controls ++(0,.5) and ++(0,-.5) .. (0.6,2.0);
    \draw[thick,  ] (1.2,0) .. controls ++(0,1) and ++(0,-1) .. (0,2.0);
    \node at (0.6,-.25) {$\scs s$};
    \node at (0,2) {$\scs \:$};
\end{tikzpicture}}
\;\; = \;\; 
\sum_{a+b+c=s-1}(-1)^b
\hackcenter{\begin{tikzpicture}[scale=0.8]
    \draw[thick, ] (0,0) -- (0,2)  node[pos=.5, shape=coordinate](DOT){};
    \draw[thick,red, double,  ] (.6,0) --  (.6,2) node[pos=.75, shape=coordinate](DOT2){};
    \draw[thick, ] (1.2,0) -- (1.2,2)  node[pos=.5, shape=coordinate](DOT1){};
    \node at (0.6,-.25) {$\scs s$};
    \node at (0,2) {$\scs \:$};
     \filldraw  (DOT) circle (2.75pt);
     \filldraw[red]  (DOT2) circle (2.75pt);
     \filldraw (DOT1) circle (2.75pt);
    \node at (-.25,1) {$\scs a$};
    \node at (.95,1.5) {$\scs E_b$};
    \node at (1.45,1) {$\scs c$};
\end{tikzpicture}}
\end{equation}


We diagrammatically represent elements of the splitter bimodules as the following trivalent diagrams
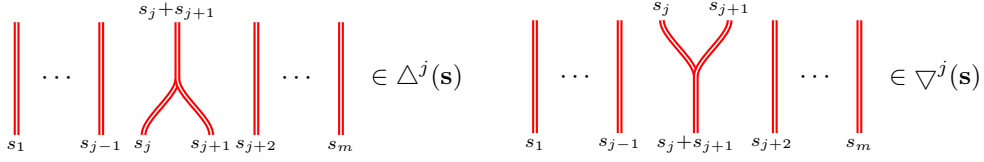
\[
\hackcenter{ \begin{tikzpicture} [scale=.75]
\draw[thick,red, double, double, ] (-2.25,0)--(-2.25,2);
\draw[thick,red, double, double, ] (-.75,0)--(-.75,2);
\draw[thick, red, double, double] (0,0).. controls ++(0,.3) and ++(0,-.3) ..(.6,1);
\draw[thick,red, double, double, ] (1.2,0).. controls ++(0,.3) and ++(0,-.3) ..(.6,1) to (.6,2);
\draw[thick,red, double, double, ] (2,0)--(2,2);
\draw[thick,red, double, double, ] (3.5,0)--(3.5,2);
 \node at (-1.5,1) {$\cdots$};
 \node at (2.75,1) {$\cdots$};
 \node at (-2.25,-.2) {$\scs s_1$};
 \node at (-.75,-.2) {$\scs s_{j-1}$};
 \node at (2,-.2) {$\scs s_{j+2}$};
 \node at (3.5,-.2) {$\scs s_m$};
 \node at (0,-.2) {$\scs s_j$};
 \node at (1.2,-.2) {$\scs s_{j+1}$};
\node at (.6,2.2) {$\scs s_j+s_{j+1}$};
\end{tikzpicture}}\in \bigtriangleup^j(\mathbf{s})
\qquad
\hackcenter{ \begin{tikzpicture} [scale=.75]
\draw[thick,red, double, double, ] (-2.25,0)--(-2.25,2);
\draw[thick,red, double, double, ] (-.75,0)--(-.75,2);
\draw[thick,red, double, double, ](0,2).. controls ++(0,-.3) and ++(0,.3) ..(.6,1);
\draw[thick,red, double, double, ](1.2,2).. controls ++(0,-.3) and ++(0,.3) ..(.6,1) to (.6,0);
\draw[thick,red, double, double, ] (2,0)--(2,2);
\draw[thick,red, double, double, ] (3.5,0)--(3.5,2);
 \node at (-1.5,1) {$\cdots$};
 \node at (2.75,1) {$\cdots$};
 \node at (-2.25,-.2) {$\scs s_1$};
 \node at (-.75,-.2) {$\scs s_{j-1}$};
 \node at (1.95,-.2) {$\scs s_{j+2}$};
 \node at (3.5,-.2) {$\scs s_m$};
 \node at (0,2.2) {$\scs s_{j}$};
 \node at (1.2,2.2) {$\scs s_{j+1}$};
\node at (.6,-.2) {$\scs s_j+s_{j+1}$};
\end{tikzpicture}}\in \bigtriangledown^j(\mathbf{s})
\]
with $n$ black strands intersecting the pictures and black (resp. $E_d$) dots putting on black (resp. red) strands. 

The tensor products of splitter bimodules are diagrammatically represented by concatenating trivalent diagrams with $n$ black strands.

\begin{example}
An element of $\bigtriangleup^j(\mathbf{s})\otimes \bigtriangledown^j(\mathbf{s})$ is represented by the diagram
$$
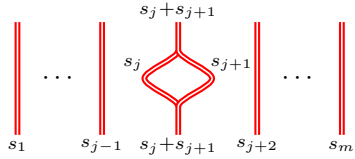
\hackcenter{ \begin{tikzpicture} [scale=.75]
\draw[thick,red, double, double, ] (-2.25,0)--(-2.25,2);
\draw[thick,red, double, double, ] (-.75,0)--(-.75,2);
\draw[thick,red, double, double, ](0.6,1.5).. controls ++(0,-.15) and ++(0,.15) ..(0,1).. controls ++(0,-.15) and ++(0,.15) ..(.6,.5);
\draw[thick,red, double, double, ] (0.6,2)to (0.6,1.5).. controls ++(0,-.15) and ++(0,.15) ..(1.2,1).. controls ++(0,-.15) and ++(0,.15) ..(.6,.5) to (.6,0);
\draw[thick,red, double, double, ] (2,0)--(2,2);
\draw[thick,red, double, double, ] (3.5,0)--(3.5,2);
 \node at (-1.5,1) {$\cdots$};
 \node at (2.75,1) {$\cdots$};
 \node at (-2.25,-.2) {$\scs s_1$};
 \node at (-.75,-.2) {$\scs s_{j-1}$};
 \node at (2.,-.2) {$\scs s_{j+2}$};
 \node at (3.5,-.2) {$\scs s_m$};
 \node at (-0.2,1.25) {$\scs s_{j}$};
 \node at (1.55,1.25) {$\scs s_{j+1}$};
\node at (.6,-.2) {$\scs s_j+s_{j+1}$};
\node at (.6,2.2) {$\scs s_j+s_{j+1}$};
\end{tikzpicture}}$$
with $n$ black strands, black dots and red dots.
An element of $\bigtriangledown^j(\mathbf{s})\otimes\bigtriangleup^j(\mathbf{s})$ is represented by the diagram
$$
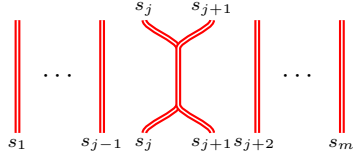
\hackcenter{ \begin{tikzpicture} [scale=.75]f
\draw[thick,red, double, double, ] (-2.25,0)--(-2.25,2);
\draw[thick,red, double, double, ] (-.75,0)--(-.75,2);
\draw[thick,red, double, double, ](0,2).. controls ++(0,-.15) and ++(0,.15) ..(.6,1.5);
\draw[thick,red, double, double, ](0,0).. controls ++(0,.15) and ++(0,-.15) ..(.6,.5);
\draw[thick,red, double, double, ](1.2,2).. controls ++(0,-.15) and ++(0,.15) ..(.6,1.5) to (.6,.5).. controls ++(0,-.15) and ++(0,.15) ..(1.2,0);
\draw[thick,red, double, double, ] (2,0)--(2,2);
\draw[thick,red, double, double, ] (3.5,0)--(3.5,2);
 \node at (-1.5,1) {$\cdots$};
 \node at (2.75,1) {$\cdots$};
 \node at (-2.25,-.2) {$\scs s_1$};
 \node at (-.75,-.2) {$\scs s_{j-1}$};
 \node at (1.95,-.2) {$\scs s_{j+2}$};
 \node at (3.5,-.2) {$\scs s_m$};
 \node at (0,2.2) {$\scs s_{j}$};
 \node at (1.2,2.2) {$\scs s_{j+1}$};
 \node at (0,-.2) {$\scs s_{j}$};
 \node at (1.2,-.2) {$\scs s_{j+1}$};
\end{tikzpicture}}
$$
with $n$ black strands, black dots and red dots. 
\end{example}

\section{A $p$-DG structure on deformed Webster algebra $D(\mathbf{s},n)$ of type $A_1$}\label{sec4}

\subsection{A $p$-derivation on the deformed Webster $p$-DG algebra $D(\mathbf{s},n)$ of type $A_1$}
We define a derivation $\partial_D:D(\mathbf{s},n)\to D(\mathbf{s},n)$ with degree $2$ by

\begin{eqnarray}
&&\partial_D(x_i)=x_i^2,\\
&&\partial_D(E(d)_ie(\mathbf{i}))=\left\{\begin{array}{ll}
(E(d)_iE(1)_i-(d+1) E(d+1)_i)e(\mathbf{i})&\text{if $d<\mathbf{i}_i\in I_{\mathbf{s}}$}\\
E(d)_iE(1)_i e(\mathbf{i})&\text{if $\mathbf{i}_i=d$}
\end{array}\right.,
\\
&&\partial_D(\psi_j e(\mathbf{i}))=\left\{
\begin{array}{ll}
-x_j \psi_j e(\mathbf{i})- \psi_j x_{j+1}e(\mathbf{i}) &\text{if $\mathbf{i}_j=\mathbf{i}_{j+1}=\mathbf{b}$},\\
s \psi_j x_{j+1} e(\mathbf{i}) &\text{if $\mathbf{i}_j=s\in I_{\mathbf{s}},\mathbf{i}_{j+1}=\mathbf{b}$},\\
\psi_j E(1)_{j+1} e(\mathbf{i}) &\text{if $\mathbf{i}_j=\mathbf{b},\mathbf{i}_{j+1}=s\in I_{\mathbf{s}}$}.
\end{array}
\right.
\end{eqnarray}

By the Leibniz rule, we extend the $p$-DG structure to the algebra $D(\mathbf{s},n)$.

We will show that this $p$-derivation $\partial_D$ preserves double and triple red, double-black relations.

\begin{theorem}
The deformed Webster algebra of type $A_1$ with the $p$-derivation $(D(\mathbf{s},n),\partial_D)$ is a $p$-DG algebra.
\end{theorem}
\begin{proof}

\noindent
$\bullet$ {Differential map}: First, we show the derivation preserves relations of $D(\mathbf{s},n)$.
It is enough to show that the derivation preserves relations \eqref{rel:r2} and \eqref{rel:r3} of $D(\mathbf{s},n)$ since it is easy to check that the derivation preserves other relations of $D(\mathbf{s},n)$.

\noindent
Relation \eqref{rel:r2}: 
We have
\begin{eqnarray}
\nonumber
\partial_D(\psi_j^2 e(\mathbf{i}))&=&
\partial_D(\psi_j e(s_j(\mathbf{i}))\psi_j e(\mathbf{i}))\\
&=&
\label{r2:cal1}
\partial_D(\psi_j e(s_j(\mathbf{i})))\psi_j e(\mathbf{i})+\psi_j e(\mathbf{i})\partial_D(\psi_j e(s_j(\mathbf{i}))).
\end{eqnarray}

When $\mathbf{i}_j=\mathbf{i}_{j+1}={\mathfrak{b}}$, $e(s_j(\mathbf{i}))=e(\mathbf{i})$.
Therefore, the element \eqref{r2:cal1} is
\begin{eqnarray}
\nonumber
&&-x_j\psi_j e(\mathbf{i})\psi_j e(\mathbf{i})-\psi_j x_{j+1}e(\mathbf{i})\psi_j e(\mathbf{i})-\psi_j e(\mathbf{i})x_j\psi_j e(\mathbf{i})-\psi_j e(\mathbf{i})\psi_j x_{j+1}e(\mathbf{i})
\\
\label{r2:cal2}
&=&-0-\psi_j x_{j+1}\psi_j e(\mathbf{i})-\psi_j x_j\psi_j e(\mathbf{i})-0
\end{eqnarray}
We have $\psi_j x_{j+1}+\psi_j x_j=x_{j+1}\psi_j +x_j\psi_j$.
Hence, the element \eqref{r2:cal2} is 
$$
-\psi_j x_{j+1}\psi_j e(\mathbf{i})-\psi_j x_j\psi_j e(\mathbf{i})=-x_{j+1}\psi_j^2  e(\mathbf{i})-x_j\psi_j^2 e(\mathbf{i})=0.
$$
This is equal to the derivation of the right-hand side.

When $\mathbf{i}_j={\mathfrak{b}}$ and $\mathbf{i}_{j+1}\in I_{\mathbf{s}}$, the element \eqref{r2:cal1} is
\begin{eqnarray}
\nonumber
&&\mathbf{i}_{j+1}\psi_j x_{j+1}e(s_j(\mathbf{i})))\psi_j e(\mathbf{i})+\psi_j e(s_j(\mathbf{i}))\psi_j E(1)_{j+1}e(\mathbf{i})\\
&=&\mathbf{i}_{j+1}x_{j}(\sum_{d=0}^{\mathbf{i}_{j+1}}(-1)^{d} x_j^{\mathbf{i}_{j+1}-d} E(d)_{j+1}) e(\mathbf{i})+ E(1)_{j+1}(\sum_{d=0}^{\mathbf{i}_{j+1}}(-1)^{d} x_j^{\mathbf{i}_{j+1}-d} E(d)_{j+1})e(\mathbf{i})
\\
\label{r2:cal3}
&=&\mathbf{i}_{j+1}x_j^{\mathbf{i}_{j+1}+1} e(\mathbf{i})
+\sum_{d=0}^{\mathbf{i}_{j+1}-1}(-1)^{d+1} x_j^{\mathbf{i}_{j+1}-d} (\mathbf{i}_{j+1}E(d+1)_{j+1}-E(1)_{j+1}E(d)_{j+1} )e(\mathbf{i})\\
\nonumber
&&+(-1)^{\mathbf{i}_{j+1}} E(1)_{j+1}E(\mathbf{i}_{j+1})_{j+1}e(\mathbf{i})
\end{eqnarray}

On the other hand, the derivation of the right-hand is
\begin{eqnarray*}
&&\sum_{d=0}^{\mathbf{i}_{j+1}}(-1)^{d}(E(1)_{j+1}E(d)_{j+1}-(d+1)E(d+1)_{j+1}) x_{j}^{\mathbf{i}_{j+1}-d} e(\mathbf{i})\\
&&+\sum_{d=0}^{\mathbf{i}_{j+1}}(-1)^{d}(\mathbf{i}_{j+1}-d)E(d)_{j+1} x_{j}^{\mathbf{i}_{j+1}-d+1} e(\mathbf{i})
\end{eqnarray*}
We find this is equal to the element \eqref{r2:cal3}.
When $\mathbf{i}_j\in I_{\mathbf{s}}$ and $\mathbf{i}_{j+1}={\mathfrak{b}}$, the derivation preserves the relation \eqref{rel:r2} by a similar calculation.

\noindent
Relation \eqref{rel:r3}: 
We show that the derivation preserves the relation \eqref{rel:r3} in the case of $\mathbf{i}_{j+1}\in I_{\mathbf{s}}, \mathbf{i}_{j}=\mathbf{i}_{j+2}={\mathfrak{b}}$.
It is easy to find the derivation preserves the relation \eqref{rel:r3} in other cases.

It is suffice to show that
$$
\partial_D(\psi_1\psi_2 \psi_1 e(\mathbf{b}s\mathbf{b})-\psi_2\psi_1 \psi_2 e(\mathbf{b}s\mathbf{b}))=
\sum_{d_1+d_2+d_3=s-1}(-1)^{d_3}\partial_D(x_1^{d_1}E(d_3)_{2}x_{3}^{d_2}e(\mathbf{b}s\mathbf{b})).
$$

The left-hand side is
\begin{eqnarray}
\nonumber
&&
\partial_D(
\psi_1e(s\mathbf{b}\mathbf{b})\psi_2 e(s\mathbf{b}\mathbf{b})\psi_1 e(\mathbf{b}s\mathbf{b})
-\psi_2e(\mathbf{b}\mathbf{b}s)\psi_1e(\mathbf{b}\mathbf{b}s) \psi_2 e(\mathbf{b}s\mathbf{b}))\\
\nonumber
&=&
\partial_D(\psi_1e(s\mathbf{b}\mathbf{b}))\psi_2 e(s\mathbf{b}\mathbf{b})\psi_1 e(\mathbf{b}s\mathbf{b})
+\psi_1e(s\mathbf{b}\mathbf{b})\partial_D(\psi_2 e(s\mathbf{b}\mathbf{b}))\psi_1 e(\mathbf{b}s\mathbf{b})
+\psi_1e(s\mathbf{b}\mathbf{b})\psi_2 e(s\mathbf{b}\mathbf{b})\partial_D(\psi_1 e(\mathbf{b}s\mathbf{b}))
\\
\nonumber
&&
-\partial_D(\psi_2e(\mathbf{b}\mathbf{b}s))\psi_1e(\mathbf{b}\mathbf{b}s) \psi_2e(\mathbf{b}s\mathbf{b})
-\psi_2e(\mathbf{b}\mathbf{b}s)\partial_D(\psi_1e(\mathbf{b}\mathbf{b}s)) \psi_2 e(\mathbf{b}s\mathbf{b})
-\psi_2e(\mathbf{b}\mathbf{b}s)\psi_1e(\mathbf{b}\mathbf{b}s) \partial_D(\psi_2 e(\mathbf{b}s\mathbf{b}))
\\
\nonumber
&=&
s x_1\psi_1\psi_2\psi_1 e(\mathbf{b}s\mathbf{b})
-x_2 \psi_1\psi_2 \psi_1 e(\mathbf{b}s\mathbf{b})
-\psi_1\psi_2 x_3 \psi_1 e(\mathbf{b}s\mathbf{b})
+\psi_1\psi_2 \psi_1E(1)_2 e(\mathbf{b}s\mathbf{b})
\\
\nonumber
&&
-E(1)_2\psi_2\psi_1 \psi_2 e(\mathbf{b}s\mathbf{b})
+\psi_2 x_1\psi_1 \psi_2 e(\mathbf{b}s\mathbf{b})
+\psi_2 \psi_1x_2 \psi_2 e(\mathbf{b}s\mathbf{b})
-\psi_2 \psi_1 s \psi_2 x_3e(\mathbf{b}s\mathbf{b}).
\end{eqnarray}

The right-hand side is

$$=\left\{
	\begin{array}{ll}
	\displaystyle \sum_{d_1+d_2+d_3=\mathbf{i}_{j+1}-1}(-1)^{d_3}x_j^{d_1}E(d_3)_{j+1}x_{j+2}^{d_2}e(\mathbf{i})
	&\text{if } \mathbf{i}_{j+1}=1, \mathbf{i}_{j}=\mathbf{i}_{j+2}={\mathfrak{b}},\\
	0
	&\text{otherwise},
	\end{array}\right. 
$$

Next, we show the derivation $\partial_D$ is $p$-nilpotent for generators.

\noindent
$\bullet$ {$p$-nilpotency}: In proof of Proposition \ref{symfun} and \ref{nilhecke}, we showed that the derivation $\partial_D$ is $p$-nilpotent for generators $x_j$, $E(d)_j$ and $\psi_j e(\mathbf{i})$ when $\mathbf{i}_j=\mathbf{i}_{j+1}=\mathfrak{b}$. 

In the case of $\mathbf{i}_j\in I_{\mathbf{s}}$ and $\mathbf{i}_{j+1}=\mathfrak{b}$, we have
$$
\partial_D^p(\psi_j e(\mathbf{i}))=\partial_D^{p-1}(s \psi_j x_{j+1} e(\mathbf{i}))=\partial_D^{p-2}(s(s+1) \psi_j x_{j+1}^2 e(\mathbf{i}))=\cdots=\prod_{a=0}^{p-1}a\psi_j x_{j+1}^p e(\mathbf{i})=0.
$$

In the case of $\mathbf{i}_j=\mathfrak{b}$ and $\mathbf{i}_{j+1}\in I_{\mathbf{s}}$, we have
$$
\partial_D^2(\psi_j e(\mathbf{i}))=\partial_D(\psi_j E(1)_{j+1} e(\mathbf{i}))=2\psi_j (E(1)_{j+1}^2-E(2)_{j+1} )e(\mathbf{i})=2\psi_j H(2)_{j+1} e(\mathbf{i}),
$$
where $H(2)_{j+1}$ is the expression of the second complete symmetric function in the elementary symmetric functions.
By the induction hypothesis, we have 
$$
\partial_D^p(\psi_j e(\mathbf{i}))=\prod_{a=1}^{p}(a+1)\psi_j H(p+1)_{j+1} e(\mathbf{i})=0,
$$
where $H(p+1)_{j+1}$ is the expression of the $p+1$-th complete symmetric function in the elementary symmetric functions.
Using Lemma \ref{lemma}, we have $\partial_D^p=0$ on the deformed Webster algebra $D(\mathbf{s},n)$.
\end{proof}

\subsection{$p$-DG structure on splitter bimodules $ \bigtriangleup^j(\mathbf{s})$ and $\bigtriangledown^j(\mathbf{s})$}

We define derivations $\partial_\bigtriangleup: \bigtriangleup^j(\mathbf{s})\to \bigtriangleup^j(\mathbf{s})$ and $\partial_\bigtriangledown(\mathbf{s}): \bigtriangledown^j(\mathbf{s})\to \bigtriangledown^j(\mathbf{s})$ by $\partial_D\otimes \mathrm{id} +\mathrm{id} \otimes\partial_D$.

By definition, these derivations have $p$-nilpotency.
We show the equality
$$
\partial_\bigtriangleup(E(i)_j\otimes 1)=\partial_\bigtriangleup(1\otimes \sum_{a+b=i}E(a)_jE(b)_{j+1})
$$
for $1\leq i\leq s_j+s_{j+1}$.

The left-hand side equals to
\begin{eqnarray*}
(E(1)_jE(i)_j-(i+1)E(i+1)_j)\otimes 1&=&1\otimes (E(1)_j+E(1)_{j+1})\sum_{a+b=i}E(a)_jE(b)_{j+1}\\
&&-1\otimes \sum_{a+b=i+1}(i+1)E(a)_jE(b)_{j+1}
\end{eqnarray*}

By a direct calculation, we find that the right-hand side equals to
\begin{eqnarray*}
&&1\otimes \sum_{a+b=i\atop a\geq 1}\partial_D(E(a)_j)E(b)_{j+1}
+1\otimes \sum_{a+b=i\atop b\geq 1}E(a)_j\partial_D(E(b)_{j+1})
\\
&=&1\otimes \sum_{a+b=i\atop a\geq 1}(E(1)_jE(a)_j-(a+1)E(a+1)_j)E(b)_{j+1}\\
&&+1\otimes \sum_{a+b=i\atop b\geq 1} E(a)_j(E(1)_{j+1}E(b)_{j+1}-(b+1)E(b+1)_{j+1})
\\
&=&1\otimes (E(1)_j+E(1)_{j+1})\sum_{a+b=i}E(a)_jE(b)_{j+1}\\
&&-1\otimes \sum_{a+b=i+1}(i+1)E(a)_jE(b)_{j+1}.
\end{eqnarray*}

By a similar way, we have the equality
$$
\partial_\bigtriangledown(1\otimes E(i)_j)=\partial_\bigtriangledown(\sum_{a+b=i}E(a)_jE(b)_{j+1}\otimes 1)
$$
for $1\leq i\leq s_j+s_{j+1}$.

Hence, we have the following statement.
\begin{theorem}
The splitter bimodules with the $p$-derivation $(\bigtriangleup^j(\mathbf{s}),\partial_\bigtriangleup)$ and $(\bigtriangledown^j(\mathbf{s}),\partial_\bigtriangledown)$ are $p$-DG modules.
\end{theorem}

\subsection{Relation to $p$-DG structure on Webster algebra}

Let $I$ be the two sided ideal generated by $\{E(d)_j,e(\mathbf{i})|1\leq j\leq m+n, d\geq 1, \mathbf{i}_1=\mathbf{b}\}$.
The quotient algebra $D({\bf s},n)/I$ is isomorphic to the Webster algebra $W_n^{\mathbf{s}}$ of type $A_1$.
Therefore, taking this quotient, the $p$-DG algebra $(D(\mathbf{s},n),\partial_D)$ naturally induces a $p$-DG structure on Webster algebra $W_n^{\mathbf{s}}$.

\begin{corollary}
The $p$-DG algebra $(D(\mathbf{s},n)/I,\partial_D)$ is isomorphic to the $p$-DG algebra $(W_n^{\mathbf{s}}, \partial_W)$ defined in Section \ref{pdg-webster}.
\end{corollary}

It is easy to find that $\partial_D$ on $D(\mathbf{s},n)/I$ is equivalent to $\partial_W$ defined in Section \ref{pdg-webster}.

\bibliographystyle{amsalpha}
\bibliography{bib_pDG}
%
%

\end{document}